\documentclass[reqno]{amsart}
\usepackage{amssymb}
\newtheorem{theorem}{Theorem}[section]

\newtheorem{lemma}{Lemma}[section]

\newtheorem{corollary}[lemma]{Corollary}
\newtheorem{proposition}[lemma]{Proposition}

\newcommand{\R}{\mbox{$\Bbb R$}} \newcommand{\C}{\mbox{$\Bbb C$}}
\newcommand{\F}{\mbox{$\Bbb F$}} 
 \def\g{\mathfrak{g}} 
  \def\i{\mathfrak{i}}
\def\l{\mathfrak{l}}  \def\n{\mathfrak{n}}
 \def\r{\mathfrak{r}} \def\s{\mathfrak{s}}
\def\t{\mathfrak{t}}

\begin{document}

\title{Minimal matrix representations of five-dimensional Lie algebras}
\maketitle

\begin{center}

 {\bf Ryad Ghanam$^{1}$  G. Thompson$^2$}\\

$^1$ Department of Mathematics, University of Pittsburgh at Greensburg,\\ Greensburg, PA 15701, U.S.A.,\\ ghanam@pitt.edu.\\
$^1$Department of Mathematics, King Fahd University of Petroleum $\&$ Minerals,\\ Dhahran 31261, Saudi Arabia,\\ghanam@kfupm.edu.sa,\\
$^2$ Department of Mathematics, University of Toledo,\\ Toledo, OH
43606, U.S.A. \\  gerard.thompson@utoledo.edu \\

\end{center}

\maketitle

\begin{abstract}
We obtain minimal dimension matrix representations for each
indecomposable five-dimensional Lie algebra over $\R$ and justify in each case that they are minimal. In each case a matrix Lie group is given
whose matrix Lie algebra provides the required representation.
\end{abstract}

\vspace{4cm}

\hspace{1cm} Key words: Lie algebra, Lie group, minimal
representation.

\hspace{1cm}Mathematics Subject Classification: 17B, 53C.

\newpage

\section{Introduction} Given a real Lie algebra $\g$ of dimension $n$ a
well known theorem due to Ado asserts that $\g$ has a faithful
representation as a subalgebra of $gl(p,\R)$ for some $p$. The
theorem does not give much information about the value of $p$ but
leads one to believe that $p$ may be very large in relation to the
size of $n$ and consequently it seems to be of limited practical
value. 
We define the invariant $\mu(\g)$
to be the minimum value of $p$. A little care must be exercised
because there may well be inequivalent representations for which
this minimum value is attained. Of course if $\g$ has a trivial
center then the adjoint representation furnishes a faithful
representation of $\g$ and in the notation used above $\mu(\g)\leq
n$. Nonetheless many algebras have non-trivial centers, nilpotent
algebras for example, and then the adjoint representation is not
faithful. Even if the center is trivial it could well be the case that $\mu(\g)\ <
n$.

Of course it is interesting to ascertain the value of $\mu$ from a theoretical point of view. However, an important practical reason is that calculations involving symbolic programs such as Maple and Mathematica
use up lots of memory when storing matrices; accordingly, calculations are likely to be faster if one can represent matrix Lie algebras using matrices of a small size.

In two recent papers \cite{GT,KB}, the problem of finding such a
minimal representation is considered for the four-dimensional Lie
algebras. In \cite{KB}, the main technique is somewhat indirect and
depends on a construction known as a left symmetric structure. In
\cite{GT}, the minimal representation has been calculated directly
without the need for considering left symmetric structures.

In this paper, we consider all indecomposable five-dimensional Lie
algebras listed in \cite{PSWZ}. It forms part of a series in which we hope to find minimal dimensional representations for all the low-dimensional algebras. Partial classifications of Lie algebras are known up to dimension nine. The five-dimensional problem is of a completely magnitude from the four-dimensional case since there are, up to isomorphism, forty classes of algebra.
It has already been shown in \cite{GST} that
$\mu(\g)\leq 5$ for each five-dimensional Lie algebra in the list.
In this paper we obtain sharper results and for the cases where
$\mu(\g) < 5$ we give an explicit representation for the Lie
algebra. This representation is given by means of a matrix Lie group, denoted
by $S$, in
local coordinates $(w,x,y,z,q)$. We also explain for those cases where $\mu(\g)\ = 5$ why we cannot have $\mu(\g) < 5$.
The matrix $S$ can be modified by adding quadratic and higher order terms to the entries without affecting the matrix Lie algebra, obtained by differentiating and evaluating at
the identity. Accordingly, as a check, we supply the corresponding right-invariant vector fields. We also adopt the convention that when referring to an ``abstract" Lie
algebra five-dimensional we use $\{e_1,e_2,e_3,e_4,e_5\}$ as a basis but when we construct a matrix representation we use $\{E_1,E_2,E_3,E_4,E_5\}$ for the corresponding
generators.

An outline of the paper is as follows. In Section 2 we give a brief overview of the indecomposable five-dimensional Lie algebras.
In Section 3 we give a two results about representations that are needed in the sequel. We shall also have occasion to use a version of Lie's Theorem for Lie algebras
over $\R$. In Section 4 we determine the few algebras where $\mu=4$ as result of the representing matrices having some partial complex structure. In Section 5 we are able to determine
all algebras for which $\mu=3$. In Section 6 we consider algebras that have a four-dimensional abelian nilradical and in Section 7 the six-dimensional nilpotent algebras. In Section 8 
we give a group matrix $S$ for each of the five-dimensional Lie algebras. In fact for the convenience of the reader we provide also an $S$-matrix for all indecomposable Lie algebras of dimension \emph{and less}. We are able to reduce the problem to the consideration of five difficult cases, namely, $A_{5,21}, A_{5,22}, A_{5,23(b\neq1)}, A_{5,31}, A_{5,38}$. These five cases are relegated to an appendix which consists of a lot of technical details. For the casual reader Section 8 is likely to be of most interest. Nonetheless, for the sake of the integrity of the results, we feel that it is essential to supply these details so that they can be verified independently. Following \cite{PSWZ} we denote each of the  five-dimensional algebras as $A_{5,k}$ where $1\leq k \leq 40$.\\

We close this Introduction by stating several conclusions. First of all one might imagine that many algebras could have representations with $\mu=4$ where the representing
matrices having some partial complex structure as necessitated by applying Lie's Theorem for Lie algebras
over $\R$ rather than over $\C$. However, relatively few such representations actually occur. Secondly, and despite the first remark, for many of the five-dimensional
indecomposable Lie algebras we do have $\mu=4$. Some of the algebras for which $\mu=5$ are $CR$-Lie algebras and over $\C$ we would have $\mu=4$, for example $A_{5,25}$
and $A_{5,26}$ which are equivalent over $\C$ to  $A_{5,19}$
and $A_{5,20}$, respectively. Also, there are some classes of algebra
depending on parameters where $\mu(\g) < 5$ \emph{but only for certain values of the parameters}. In fact, as a crude count, $21$ of the $40$ algebras have $\mu=5$, although
$\mu(\g) < 5$ for several of those 21 algebras for special values of the parameters.

R. Ghanam would like to acknowledge the support of King Fahd
University of Petroleum and Minerals through the funded project
number FT101010. In this paper the many calculations were performed with the help of the symbolic manipulation program Maple.


\section{Real and complex five-dimensional Lie algebras}

The \emph{real} five-dimensional indecomposable  Lie algebras were
classified by G Mubarakzyanov \cite{Mub1}. They can be found easily
in \cite{PSWZ} and we list them in Section 4. The first six algebras
are nilpotent. These six algebras are distinguished by their index
of nilpotence and the dimension of the derived algebra except for
the filiforms $A_{5,2}$ and $A_{5.6}$. However, $A_{5.2}$ has a
codimension one abelian ideal, whereas $A_{5,6}$ does not so the six
algebras are mutually non-isomorphic.

We remark that the algebras $A_{5,7}-A_{5,18}$ have a
four-dimensional abelian nilradical; $A_{5,19}-A_{5,29}$ have a
four-dimensional non-abelian nilradical isomorphic to
$H\bigoplus\R$, where $H$ denotes the three-dimensional Heisenberg
algebra; $A_{5,30}-A_{5,32}$ have a nilradical that is isomorphic to
the unique four-dimensional indecomposable nilpotent algebra;
$A_{5,33}-A_{5,35},A_{5,38},A_{5,39}$ have an abelian
three-dimensional nilradical and $A_{5,36},A_{5,37}$ have a
three-dimensional nilradical that is isomorphic to $H$; $A_{5,40}$
is the only algebra that is not solvable, that is, has a non-trivial
Levi decomposition, being a semi-direct product of $\s\l(2,\R)$ and
the abelian algebra $\R^2$. In particular, all of these algebras,
with the exception of $A_{5,36},A_{5,37}$ and $A_{5,40}$ have a
three-dimensional abelian subalgebra.

Given a Lie algebra $\g$ we denote its derived algebra by $[\g,\g]$.
If now $\g$ is an indecomposable five-dimensional Lie algebra the
dimension of $[\g,\g]$ is precisely one only for the Heisenberg
algebra $A_{5,4}$ and is two in just $A_{5,1}$ and $A_{5,5}$ which
are both nilpotent. \emph{Apart from the non-solvable $A_{5,40}$,
for all other indecomposable five-dimensional indecomposable
algebras the dimension of $[\g,\g]$ is either three or four}. In
fact in the following algebras the dimension of $[\g,\g]$ is three:
$A_{5,2},A_{5,3},A_{5,6},A_{5,8},A_{5,10},A_{5,14},A_{5,15(a=0)},A_{5,19(a=1)},A_{5,20(a=0,1)},A_{5,22},A_{5,26(p=0)},A_{5,27},A_{5,28(a=1)},\\A_{5,29},A_{5,30(a=1)},A_{5,32}-A_{5,39}$.

\section{Two representation results}

\subsection{Transposition about the anti-diagonal.}

\begin{proposition} Suppose that the Lie algebra $\g$ has a
representation as a subalgebra of $\g\l(p,\R)$. Suppose that
$L:\g\l(p,\R) \rightarrow \g\l(p,\R)$ is a (linear) involution, that
is, has period two. Then mapping a representing matrix $M$ to
$-LM^tL$ gives a second inequivalent representation of $\g$.
\end{proposition}

\begin{proof} Given $M \in \g\l(p,\R)$ a representing matrix for
$\g$, map it to $\phi(M)=-LM^tL$. For a second such matrix $N$ we
have
$[\phi(M),\phi(N)]=[-LM^tL,-LN^tL]=[LM^tLLN^tL-LN^tLLM^tL]=[LM^tN^tL-LN^tM^tL]=L[N,M]^tL=\phi([M,N])$.
\end{proof}

\begin{corollary} Suppose that the Lie algebra $\g$ has a
representation as a subalgebra of $\g\l(p,\R)$. Then transposing the
representing matrices about the anti-diagonal (and taking negatives)
gives a second inequivalent (that is not necessarily equivalent)
representation of $\g$. \end{corollary} \begin{proof} Take for $L$
in the Proposition the matrix whose only non-zero entries are $1$'s
down the anti-diagonal. Then the map $\phi$ consists of taking a
negative and transposing about the anti-diagonal. \end{proof}
\noindent We use the preceding Corollary in the sequel to reduce
some of the subcases that are needed to investigate whether a
particular type of representation exists.

\subsection{}

The following Theorem \cite{RT} explains how to find a
low-dimensional representation for a solvable algebra that has an
abelian nilradical with abelian complement.

\noindent \begin{theorem} \label{ab_rep_thm} Suppose that the
$n$-dimensional Lie algebra $\g$ has a basis $\{e_1,e_2,...,e_n\}$
and only the following non-zero brackets: $[e_a,e_i]=C^j_{ai}e_j$,
where $(1 \leq i,j \leq r$, $r+1 \leq a,b,c \leq n)$.  Suppose that
$\g$ has an abelian nilradical for which a basis is
$\{e_1,e_2,...,e_r\}$ and $\{e_{r+1},...,e_n\}$ is a basis for an
abelian subalgebra complementary to the nilradical. Then $\g$ has a
faithful representation as a subalgebra of $gl(r+1, \Bbb R)$.
\end{theorem} \noindent As result of Theorem 3.1 we can find
representations for algebras $A_{5,33,34,35}$ in $\g\l(4,\R)$.

\section{Extension of Lie's Theorem over $\R$.}

In this paper we are concerned with \emph{real} Lie algebras. Lie's
Theorem is usually stated for complex Lie algebras; in adapting
Lie's Theorem to $\R$ one has to allow for $2\times 2$ complex
blocks along the diagonal. As such, if a solvable Lie algebra has a
representation in $\g\l(4,\R)$ the matrices may be assumed to be one
of one of the following five forms:
\begin{eqnarray} \label{Lie}
 & (i)  \, \left[\begin{smallmatrix}
a&b&c&d\\ -b&a&e&f\\ 0&0&g&h\\ 0&0&-h&g \end{smallmatrix}\right] \,
(ii)  \, \left[\begin{smallmatrix} a&b&c&d\\ -b&a&e&f\\ 0&0&g&h\\
0&0&0&i \end{smallmatrix}\right]  \, (iii)  \,
\left[\begin{smallmatrix} a&b&c&d\\ 0&e&f&g\\ 0&0&h&i\\ 0&0&-i&h
\end{smallmatrix}\right]  \, (iv)  \, \left[\begin{smallmatrix}
a&b&c&d\\ 0&e&f&g\\ 0&-f&e&h\\ 0&0&0&i \end{smallmatrix}\right]  \,
(v)  \, \left[\begin{smallmatrix} a&b&c&d\\ 0&e&f&g\\ 0&0&h&i\\
0&0&0&j \end{smallmatrix}\right]. \end{eqnarray} We can exclude the
third case in view of Corollary 3.2. In order to simplify a given
representation it makes sense to use only a \emph{non-singular}
matrix of the same type as one of these five. It appears as though
we have to consider many cases for representing five-dimensional
algebras in $\g\l(4,\R)$. However, it is remarkable that cases
(i-iv) correspond to representations for only a very limited number
of and, in some cases, unexpected algebras. In the interests of
concision we shall be content merely to sketch the results.

Notice first of all that it follows from Section 3.3 that if a
five-dimensional algebra $\g$ has a representation of type (i)-(iv)
then at least one of its adjoint matrices must have non-real
eigenvalues. Actually this remark is not obvious but is nonetheless
true as we shall explain. In any case given the last remark, we can
proceed as follows. First of all, the nilpotent algebras
$A_{5,1}-A_{5,6}$ cannot have representations of types (i)-(iv).
Secondly, in Section 6 we give a systematic analysis of algebras
$A_{5,7}-A_{5,18}$ so we leave them to one side for the moment.
Thirdly, of algebras $A_{5,19}-A_{5,39}$ only
$A_{5,25},A_{5,26},A_{5,35},A_{5,37}$ and $A_{5,39}$ have adjoint
matrices with non-real eigenvalues. As regards $A_{5,35}$ we know
that $\mu=4$ in view of Theorem 3.1. Similarly we will show that
$\mu=4$ for $A_{5,37}$ in Section 5. As for $A_{5,39}$ it is
equivalent as a complex algebra to $A_{5,38}$ and the latter will be
be shown to have $\mu=5$ so that $\mu=5$ for $A_{5,39}$ also. Hence
of algebras $A_{5,19}-A_{5,39}$ whose adjoint matrices have non-real
eigenvalues, $\mu$ is in doubt only for $A_{5,25}$ and $A_{5,26}$.

\begin{proposition}
For $A_{5,25}$ and $A_{5,26}$ we have $\mu=5$.
\end{proposition}
\begin{proof}
Assuming that $\mu=4$ we shall obtain a contradiction for each of cases (i-v).
First of all case (v) is excluded because all of the adjoint matrices would have real eigenvalues which is not true for $A_{5,25}$ and $A_{5,26}$. As regards case (i) for a Lie algebra $\g$, the derived algebra $[\g,\g]$ must consist of matrices of the form
$\left[\begin{smallmatrix} 0&0&*&*\\ 0&0&*&*\\ 0&0&0&0\\ 0&0&0&0
\end{smallmatrix}\right]$ where only the asterisks can be non-zero.
It must be abelian and hence $\g$ cannot be $A_{5,25}$ or $A_{5,26}$. Case (iii) can be excluded in view of Proposition 3.1 once we have ruled out case (ii) which we look at next. The derived
algebra must consist of matrices of the form $\left[\begin{smallmatrix} 0&0&*&*\\ 0&0&*&*\\
0&0&0&*\\ 0&0&0&0 \end{smallmatrix}\right]$ where only the asterisks
can be non-zero. Let us suppose first of all that we have one
non-nilpotent matrix $E$ of type (ii). Then we obtain a
six-dimensional solvable algebra that has the nilpotent algebra
$A_{5,1}$ as its nilradical and $E$ spans a complement to it. In this six-dimensional algebra we find that
$ad\, E$ has eigenvalues $\{0, g-i, a-i \pm b \sqrt{-1}, a-g \pm b\
\sqrt{-1}\}$. The restriction of $ad\, E$ to the five-dimensional
subalgebra must have eigenvalues that are a subset of these six and
of course $0$ must be one of them. The only possibility is to remove $g-i$
which leaves the eigenvalues $\{0, a-i \pm b \sqrt{-1}, a-g \pm b\
\sqrt{-1}\}$ and we can only have $A_{5,17prs}$ or $A_{5,18}$. In
fact $A_{5,17 prs(s=1)}$ and $A_{5,18}$ do occur and examples of
each are listed in Section 10. However, $A_{5,25}$ and $A_{5,26}$ are excluded because their ad-matrices have at most a pair of complex eigenvalues and not two pairs.

If there are two non-nilpotent matrices of type (ii), $E_6$ and $E_7$ say, then $[\g,\g]$
is three-dimensional which as regards $A_{5,25}$ and $A_{5,26}$
would only allow $A_{5,26},p=0$. We shall obtain together with the representation of $A_{5,1}$ a
seven-dimensional codimension two nilradical solvable algebra and
the question now is whether there exists a five-dimensional
subalgebra besides $A_{5,1}$. Such an algebra $\g$ must have
$[\g,\g]$ of dimension three. Again letting $E_6$ be the given
matrix of type (ii) in \ref{Lie} we find that the eigenvalues of
$ad\, E$ are $\{0,0, i-g, a-g \pm b \sqrt{-1}, a-i \pm b\
\sqrt{-1}\}$. Now comparing with the adjoint matrices of $A_{5,26},p=0$ we deduce that $a=g=i,b=1$. In fact   arbitrary multiples of the identity can be added to $E_6$ and $E_7$ without affecting the putative representation that we are looking for. However the same argument that applies to $E_6$ applies equally to $E_7$ which would imply that a linear combination of $E_6$ and $E_7$ is nilpotent, a contradiction.

Case (iv) is similar to case (ii). Now $[\g,\g]$ consist of matrices of the form
$\left[\begin{smallmatrix} 0&*&*&*\\ 0&0&0&*\\ 0&0&0&*\\ 0&0&0&0
\end{smallmatrix}\right]$ where only the asterisks can be non-zero.
Let us suppose first of all that we have one non-nilpotent matrix
$E$  of type (iv). Then we obtain a six-dimensional solvable algebra
that has the nilpotent algebra $A_{5,4}$ as its nilradical and $E$
spans a complement to it. However, in the latter algebra we find
that $ad\, E$ has eigenvalues $\{0, a-i, a-e \pm f \sqrt{-1}, e-i \pm
f\ \sqrt{-1}\}$. The restriction of $ad\,E$ to $\g$
must have eigenvalues that are a
subset of these six and again we find two pairs of non-real complex conjugates.

If there are two non-nilpotent matrices $E_6$ and $E_7$ of type (iv) we shall obtain
together with the representation of $A_{5,4}$ a seven-dimensional
codimension two nilradical solvable algebra. Such an algebra $\g$ must have $[\g,\g]$ of dimension
three. Once again we find that, matching adjoint matrices to $A_{5,26},p=0$, a linear combination of $E_6$ and $E_7$ is nilpotent, a contradiction.
\end{proof}

\section{Representations in $\g\l(3,\R)$}

Let $\g$ be an indecomposable real solvable five-dimensional
algebra. Clearly $\g$ cannot have a representation as a subalgebra
of $\g\l(2,\R)$ so the smallest value of $n$ for which $\g$ can be
represented in $\g\l(n,\R)$ is $n=3$.

\begin{lemma} Any three-dimensional abelian subalgebra of
$\g\l(3,\R)$ contains a multiple of the identity.
\label{first_lemma} \end{lemma}
\begin{proof} Consider an element of such an abelian subalgebra. We
can put it into one of the following four forms by a \emph{real}
change of basis:
\[  \begin{array}{llll}
 & a) \left[\begin{matrix}
\lambda&0&0\\ 0&\mu&0\\ 0&0&\nu\\ \end{matrix}\right]

\quad b) \left[\begin{matrix} \lambda&1&0\\ 0&\lambda&1\\
0&0&\lambda\\ \end{matrix}\right]

 \quad c) \left[\begin{matrix}
\lambda&0&1\\ 0&\mu&0\\ 0&0&\lambda\\ \end{matrix}\right]

 & d) \left[\begin{matrix} \alpha&\beta&0\\
-\beta&\alpha&0\\ 0 &0&\gamma\\ \end{matrix}\right] (\beta\neq0) .
\end{array} \] \noindent

The proof follows easily now by considering the centralizer in each of these four cases.
\end{proof}

\begin{corollary} The only solvable five-dimensional indecomposable
algebras that can be represented as a subalgebra of $\g\l(3,\R)$ are
$A_{5,36}$ or $A_{5,37}$. \end{corollary} \begin{proof} If a
subalgebra of $\g\l(3,\R)$ contains $I$ it will be decomposable.
However, the only algebras that have do not have an abelian
three-dimensional subalgebra are $A_{5,36}$ and
$A_{5,37}$.\end{proof}

As regards $A_{5,37}$ let us note that it is equivalent over $\C$ to
$A_{5,36}$ which means that both are equivalent considered as
algebras over $\C$. To see how, make a change of basis according to
$$\overline{e_1}=-\frac{e_1}{2}, \overline{e_2}=\frac{e_2+e_3}{2},
\overline{e_3}=\frac{\sqrt{-1}(e_2-e_3)}{2}, \overline{e_4}=2e_4+e_5,
\overline{e_5}=-\sqrt{-1} e_5.$$ We obtain the following brackets, which are
formally identical to $A_{5,36}$ except for the factor of $\sqrt{-1}$ which
can then be removed by scaling $e_1$: $$[e_2,e_3]=\sqrt{-1} e_1,
[e_1,e_4]=e_1, [e_2,e_4]=e_2, [e_2,e_5]=-e_2, [e_3,e_5]=e_3.$$

The Lie algebra $A_{5,36}$ does have a representation in
$\g\l(3,\R)$: it is isomorphic to the space of trace-free upper
triangular matrices.

Since $A_{5,37}$ is equivalent over $\C$ to $A_{5,36}$ it certainly
has a a representation in $\g\l(3,\C)$. However, we claim that
$A_{5,37}$ does not have a representation in $\g\l(3,\R)$. Indeed if
it did by Lie's Theorem it would have a representation by matrices
of the form $\left[\begin{smallmatrix} a&b&c\\
 -b& a &d\\
 0&0&e
\end{smallmatrix} \right]$ or $\left[\begin{smallmatrix} c&d&e\\
 0& a &b\\
 0&-b&a
\end{smallmatrix} \right].$ However, in either case the derived
algebra is at most two-dimensional whereas for $A_{5,37}$ it is
three-dimensional so $A_{5,37}$ cannot be represented in
$gl(3,\R)$.

We remark finally that algebra $A_{5,40}$ is the Lie algebra of the
special affine group and so by its very definition has a
representation in $gl(3,\R)$ but not an upper triangular triangular
representation. Thus we have determined those algebras that can be
represented in $gl(3,\R)$ and $\g\l(3,\C)$.

\section{Four-dimensional abelian nilradical algebras}


Now we consider algebras $\g$ $A_{5,7}-A_{5,18}$ for which the
algebra is solvable but not nilpotent and for which $\n\i\l(\g)$ is
abelian. We quote next a result of Schur-Jacobson \cite{Jac}.

\begin{proposition}(Schur-Jacobson) The maximal \emph{commutative}
subalgebra of $\g\l(n,\R)$ is of dimension $1+[\frac{n^2}{4}]$ where
$[]$ denotes the integer part of a real number. Up to change of
basis if $n$ is even the subalgebra consists of the upper left hand
block with row entries running from $1+\frac{n}{2}$ to $n$ and
column entries from $1$ to $\frac{n}{2}$ together with multiples of
the identity; if $n$ is odd the subalgebra consists of the upper
left hand blocks with row entries running from either
$\frac{n+1}{2}$ to $n$ and column entries from $1$ to
$\frac{n-1}{2}$ or $\frac{n+3}{2}$ to $n$ and $1$ to
$\frac{n+1}{2}$, respectively, together with multiples of the
identity. \end{proposition}

\noindent We apply the Proposition in the case $n=4$. We shall not
want to include  multiples of the identity because it will lead to a
decomposable algebra. Accordingly we assume that we have a basis for
$\n\i\l(\g)$ of the following form:
\[  \begin{array}{llll}
 & E_1 = \left[\begin{smallmatrix}
0&0&0&1\\ 0&0&0&0\\ 0&0&0&0\\ 0&0&0&0 \end{smallmatrix}\right] ,\,
E_2 = \left[\begin{smallmatrix} 0&0&1&0\\ 0&0&0&0\\ 0&0&0&0\\
0&0&0&0 \end{smallmatrix}\right] ,\, E_3 = \left[\begin{smallmatrix}
0&0&0&0\\ 0&0&0&1\\ 0&0&0&0\\ 0&0&0&0 \end{smallmatrix}\right]  ,\,
E_4 = \left[\begin{smallmatrix} 0&0&0&0\\ 0&0&1&0\\ 0&0&0&0\\
0&0&0&0 \end{smallmatrix}\right]. \end{array}  \]


To obtain a full basis for $\g$ we add a generator $E_5$ of the form
$E_5=\left[\begin{smallmatrix} a&b&0&0\\ c&d&0&0\\ 0&0&e&f\\ 0&0&g&h
\end{smallmatrix}\right]$. This form of $E_5$ is dictated by the
requirement that $ad E_5$ map $\n\i\l(\g)$ spanned by
$E_1,E_2,E_3,E_4$ to itself. Then the Jacobi identity is satisfied
and we have the following brackets: \begin{eqnarray} && [E_1,E_5]=
\nonumber (h-a)E_1+gE_2-cE_3,[E_2,E_5]=fE_1+(e-a)E_2-cE_4,\\
\nonumber && [E_3,E_5]=-bE_1+(h-d)E_3+gE_4, \nonumber
[E_4,E_5]=-bE_2+fE_3+(e-d)E_4. \end{eqnarray}

Now consider $ad\, E_5$. Its bottom row and last column are zero. We look at the
upper right $4 \times 4$ block given by $M=\left[\begin{smallmatrix}
h-a&f&-b&0\\ g&e-a&0&-b\\ -c&0&h-d&f\\
0&-c&g&e-d\end{smallmatrix}\right]$ and consider its possible Jordan
normal form: for Lie algebras with abelian codimension one
nilradical the Jordan normal form is a complete invariant apart from
an overall scaling. The matrix $M$ enjoys a particular property

\begin{lemma}
If $\lambda_1,\lambda_2,\lambda_3,\lambda_4$ are the eigenvalues of $M$ there is an ordering of them so that $\lambda_1+\lambda_2=\lambda_3+\lambda_4$.
\end{lemma}
\begin{proof}
In fact solving explicitly we find that the eigenvalues are of the form $A \pm \sqrt{B \pm \sqrt{C}}$   where $A,B,C$ are determined in terms of $a,b,c,d,e,f,g,h$.
\end{proof}

 We compare with algebras $A_{5,7}-A_{5,18}$ and
obtain various of these algebras sometimes for special values of the
parameters. It is too much to be able to find the Jordan normal
forms for the matrix in the form above. However, going back to the
the original form of the algebra we can make a change of basis of
the form $S=\left[\begin{smallmatrix} P&0\\
0&Q\\\end{smallmatrix}\right]$ where $P$ and $Q$ are arbitrary
non-singular matrices. More precisely we multiply each of the
matrices $E_i, (1 \leq i \leq 5)$ on the left by $S^{-1}$ and on the
right by $S$; although each of the $E_i, (1\leq i \leq 4)$ are
changed we can in effect use the original $E_i$ as part of the basis
since the $E_i$'s are transformed into linear combinations of
themselves. As regards the matrix $E_5$, its two blocks are
conjugated separately by $P$ and $Q$. As such we may assume that
each of these blocks are in Jordan normal form. There are three
Jordan normal forms for a $2 \times 2$ matrix giving nine cases in
toto. We summarize these cases as follows in which we give first of
all the conditions on $a,b,c,d,e,f,g,h$ to be in one of the nine
cases, then the eigenvalues $\lambda$ of the matrix $ad(E_5)$ and
finally conditions on the parameters of the algebras
$A_{5,7}-A_{5,18}$ so that they can have a representation in
$\g\l(4,\R)$.

\begin{enumerate}

\item $c=-b, d=a, g=-f, h=e, (bf\neq0): \lambda = e - a \pm ib \pm
if$

 $A_{5,13apr}$ $g=a=p=1, f=b=\frac{r}{2}$, $A_{5,17}$ $e-a=p=r, b+f=1, b-f=s$

\item $c=-b, d=a, f=0, g=0, (b\neq0): \lambda = e - a \pm ib, h - a
\pm ib $

$A_{5,17prs}$ $e-a=p, h-a=q, b=s=1$

\item $c=-b, d=a, h=e, f=1, g=0, (b\neq0): \lambda = e - a \pm ib, e
- a \pm ib$

$A_{5,18p}$ $e-a=p, b=1$

\item  $c=0, b=0, g=-f, h=e, (f\neq0):\lambda = e - a \pm if,  e - d
\pm if$

$A_{5,17prs}$ $e-a=p, e-d=q, f=s=1$

\item $c=0, b=0, f=0, g=0: \lambda = e - a,  e - d, h-a, h-d$

$A_{5,7abc}$ $e-a=1, e-d=\overline{a}, h-a=\overline{b}
,h-d=\overline{a}$ so $\overline{a}+\overline{b}-\overline{c}=1$

\item $c=0, b=0, h=e, f=1, g=0: \lambda = e - a,  e - a, e - d,  e
-d$

$A_{5,15a}$ $e-a=1, e-d=\overline{a}$

\item  $d=a, c=0, b=1, g=-f, h=e, (f\neq0): \lambda = e - a \pm if,
e - a \pm if$

$A_{5,18p}$ $e-a=p, f=1$

\item $d=a, c=0, b=1, f=0, g=0: \lambda = e - a,  e - a, h-a, h-a$

$A_{5,15a}$ $e-a=1, h-a=\overline{a}$

\item $d=a, c=0, b=1, h=e, f=1, g=0: \lambda = e - a,  e - a, e-a,
e-a$

$A_{5,11c}$ $e-a=1$ so $\overline{c}=1$.

\end{enumerate}

\noindent To summarize: in $A_{5,7}-A_{5,18}$, the algebras
$A_{5,8c},A_{5,9bc},A_{5,10},A_{5,12},A_{5,14p}$ and $A_{5,16pq}$
do not occur at all as subalgebras of $\g\l(4,\R)$. Algebras
$A_{5,15a}$ and $A_{5,18p}$ do occur and $A_{5,7abc}$ occurs but only for
the case $a+b-c=1$, $A_{5,11c}$ for $c=1$, $A_{5,13apr}$ for $a=p=1$  and $A_{5,17}$ for $p=r$ or
$s=1$.


\section{Five-dimensional nilpotent algebras}

Now suppose that $\g$ is nilpotent and can be represented in
$\g\l(4,\C)$. We note that both the algebras $A_{5,1}$ and $A_{5,2}$
have a four-dimensional abelian subalgebra indeed ideal. For these
algebras we can proceed much as we did for the abelian nilradical
case except now we must have that $ad\, E_5$ is nilpotent. We note
first of all that we can add a multiple of the identity to $E_5$ and
it is does not change $ad\,E_5$ so we may assume that $E_5$ has
trace zero and so we put $h=-(a+d+e)$. Next we shall demand that
$ad\,E_5$ has trace zero which gives that $d=-a$ and then we find
that the trace of $(ad\,E_5)^3$ is  zero. In order to make $ad\, E_5$
nilpotent it is sufficient to have that the trace of $(ad\,E_5)^2$ and
the determinant of $ad\,E_5$ zero for then $ad\,E_5$ will have all
eigenvalues zero. These conditions give us $a^2+bc=0$ and
$e^2+fg=0$. In order for $E_5$ not to vanish entirely and using the
block change of basis we can reduce to the cases where only $b=1$ or
$f=1$ are the only non-zero entries or else $b=f=1$. The first two
of these cases correspond to $A_{5,1}$ whereas the third after
making a change of basis $$e_1^{\prime}=\frac{e_1-e_2}{2},
e_2^{\prime}=e_2, e_3^{\prime}=-\frac{e_3}{2} ,e_4^{\prime}=e_1+e_4
, e_5^{\prime}=e_5$$ becomes $[e_1,e_5]=e_2,[e_3,e_5]=e_1$. This
algebra is decomposable being a direct sum of $A_{4,1}$ and $\R$. In
particular for $A_{5,2}$ we cannot have $\mu=4$.

Now we consider algebras $A_{5,3}$ and $A_{5,6}$. For $A_{5,3}$  we
have non-zero brackets $[e_3,e_4] = e_2$, $[e_3,e_5] = e_1$,
$[e_4,e_5] =e_3$. The first two brackets give us a copy of
$A_{5,1}$; indeed if we permute $e_3$ and $e_5$ and then change the
sign of $e_5$ we obtain $[e_3,e_5] = e_1, [e_4,e_5] = e_2, [e_3,e_4]
=e_5$. Now the argument above gave us, up to change of basis, a
unique representation of $A_{5,1}$ in $\g\l(4,\R)$. Since $[e_3,e_4]
\neq e_5$ we conclude that there is no representation of $A_{5,3}$
in $\g\l(4,\R)$. Similarly for $A_{5.6}$ $[e_3,e_4]=e_1,
[e_2,e_5]=e_1$, $[e_3,e_5]=e_2, [e_4,e_5]=e_3$ the latter three
brackets produce a copy of $A_{5.2}$. However, we have shown that
there is no representation of $A_{5,2}$ in $\g\l(4,\R)$ and
therefore not for $A_{5,6}$, either.

\section{Group representations corresponding to Lie algebras in Dimension Five and less}

\noindent $A_{2.1}$ $[e_1,e_2]=e_2$:
\[ S=\left[ \begin{matrix}
e^{x} &y \\
 0 & 1
\end{matrix} \right]. \]
\noindent Right-invariant vector fields: $-(D_x+yD_y),D_y$.


\noindent  $A_{3.1}$ $[e_2,e_3]=e_1$:
\[ S=\left[ \begin{matrix}
1 &x&z \\
 0 & 1  &y \\
 0&0&1
\end{matrix} \right]. \]
\noindent Right-invariant vector fields: $D_z,D_y,D_x+yD_z$.\\

\noindent $A_{3.2} [e_1,e_3]=e_1, [e_2,e_3]=e_1+e_2$:
\[ S=\left[ \begin{matrix}
e^{z}& ze^{z}&x\\
0& e^{z} &y \\
 0&0&1
\end{matrix} \right]. \]
\noindent Right-invariant vector fields: $D_x,D_y,D_z+(x+y)D_x+yD_y$.\\

\noindent $A_{3.3},A_{3.4},A_{3.5a}$ $[e_1,e_3]=e_1,[e_2,e_3]=ae_2$:
\[ S=\left[ \begin{matrix}
e^{z}& 0&x\\
0& e^{az} &y \\
 0&0&1
\end{matrix} \right]. \]
\noindent Right-invariant vector fields: $D_x,D_y,D_z+xD_x+ayD_y$.\\

\noindent $A_{3.6},A_{3.7a} [e_1,e_3]=ae_1-e_2,[e_2,e_3]=e_1+ae_2$:
\[ S=\left[ \begin{matrix}
e^{az}\cos z & e^{az}\sin z&x\\
-e^{az}\sin z& e^{az}\cos z&y \\
 0&0&1
\end{matrix} \right]. \]
\noindent Right-invariant vector fields: $D_x,D_y,D_z+(ax+y)D_x+(ay-x)D_y$.\\

\noindent $A_{3.8}$ $[e_1,e_3]=-2e_2, [e_1,e_2]=
e_1,[e_2,e_3]=e_3$:
\[S = \left[ \begin {array}{cc} \cosh  x  +\sinh  x
  \cosh y  &- {e^{-z}} \sinh x  \sinh
y \\\noalign{\medskip} {e^{z}} \sinh  x  \sinh
  y &\cosh x  -\sinh  x
 \cosh y  \end {array} \right] \]
\noindent Right-invariant vector fields:
$\frac{e^{z}}{2}(\sinh y D_x-\frac{(\cosh x \cosh y-\sinh x)}{\sinh x }D_y-\frac{(\sinh x \cosh y-\cosh x)}{(\sinh y \sinh x)}D_z),\\
\frac{1}{2}(\cosh y D_x-\frac{\cosh x \sinh y}{\sinh x}D_y-D_z),
\frac{e^{-z}}{2}(-\sinh y D_x+\frac{(\cosh x\cosh y+\sinh x)}{\sinh x}D_y+\frac{(\sinh x\cosh y+\cosh x)}{(\sinh y \sinh x)}D_z))$.\\

\noindent $A_{3.9}$ $[e_1,e_2]=e_3,[e_2,e_3]= e_1 ,[e_3,e_1]= e_2$:
 \[ S =\left[ \begin {array}{ccc} \cos  x \cos  y
 \cos  z  -\sin  x  \sin z
 &\sin x  \cos  y  \cos  z
  +\cos  x  \sin  z  &-\sin y
 \cos  z  \\\noalign{\medskip}-\cos  x
 \cos  y  \sin z -\sin  x
 \cos z  &-\sin  x  \sin z
  \cos  y  +\cos  x  \cos  z
 &\sin  y  \sin z
\\\noalign{\medskip}\cos x  \sin  y  &\sin
 x  \sin  y  &\cos y
\end {array} \right]. \]
\noindent Right-invariant vector fields: $D_z,
\frac{\sin z}{\sin y}D_x+\cos z D_y-\frac{\cos y\sin z}{\sin y}D_z,
\frac{\cos z}{\sin y }D_x-\sin z D_y-\frac{\cos y \cos z}{\sin y}D_z$.


\noindent $A_{4.1}$ $[e_2,e_4]=e_1, [e_3,e_4]=e_2$:
\[ S=\left[ \begin{matrix}
1 &w& \frac{w^2}{2} & x \\
0 & 1 &w& y\\
0 & 0 & 1 &z \\
0&0&0&1
\end{matrix} \right]. \]
\noindent Right-invariant vector fields: $D_x,D_y,D_z,D_w+yD_x+zD_y$.\\

\noindent $A_{4.2a} (a \neq 0)$ $[e_1,e_4]=ae_1, [e_2,e_4]=e_2, [e_3,e_4]=e_2+e_3$:
\[ S=\left[ \begin{matrix}
e^{aw} &0& 0 & x \\
0 & e^{w} &we^{w}& y\\
0 & 0 & e^{w} &z \\
0&0&0&1
\end{matrix} \right]. \]
\noindent Right-invariant vector fields: $D_x,D_y,D_z,D_w+axD_x+(y+z)D_y+zD_z$.\\

\noindent $A_{4.3}$ $[e_1,e_4]=e_1, [e_3,e_4]=e_2$:
\[ S=\left[ \begin{matrix}
e^{w} &0& 0 &x \\
0 & 1 &w& y\\
0 & 0 & 1 &z \\
0&0&0&1
\end{matrix} \right]. \]
\noindent Right-invariant vector fields: $D_x,D_y,D_z,D_w+xD_x+zD_y$.\\

\noindent $A_{4.4}$ $[e_1,e_4]=e_1, [e_2,e_4]=e_1+e_2, [e_3,e_4]=e_2+e_3$:
\[ S=\left[ \begin{matrix}
e^{w} &we^{w}& \frac{w^2}{2}e^{w} &x \\
0 & e^{w} &we^{w}& y\\
0 & 0 & e^{w} &z \\
0&0&0&1
\end{matrix} \right]. \]
\noindent Right-invariant vector fields: $D_x,D_y,D_z,D_w+(x+y)D_x+(y+z)D_y+zD_z$.\\

\noindent $A_{4.5ab}$ $(0 \leq ab, \quad -1 \leq a \leq b \leq 1)$ $[e_1,e_4]=e_1, [e_2,e_4]=ae_2, [e_3,e_4]=be_3$:
\[ S=\left[ \begin{matrix}
e^{w} &0& 0 &x \\
0 & e^{aw} &0& y\\
0 & 0 & e^{bw} &z \\
0&0&0&1
\end{matrix} \right]. \]
\noindent Right-invariant vector fields: $D_x,D_y,D_z,D_w+xD_x+ayD_y+bzD_z$.\\

\noindent $A_{4.6ab}$ $(a \neq 0, b \geq 0)$ $[e_1,e_4]=ae_1, [e_2,e_4]=be_2-e_3, [e_3,e_4]=e_2+be_3$:
\[ S=\left[ \begin{matrix}
e^{aw} &0& 0 &x \\
0 & e^{bw}\cos w &e^{bw}\sin w& y\\
0 & -e^{bw}\sin w & e^{bw}\cos w &z \\
0&0&0&1
\end{matrix} \right]. \]
\noindent Right-invariant vector fields: $D_x,D_y,D_z,D_w+axD_x+(by+z)D_y+(bz-y)D_z$.\\

\noindent $A_{4.7}  [e_2,e_3]=e_1, [e_1,e_4]=2e_1, [e_2,e_4]=e_2, [e_3,e_4]=e_2+e_3 $:
\[ S=\left[ \begin{matrix}
e^{2w} &-ze^{w}& (y-zw)e^{w} &x \\
0 & e^{w} &we^{w}& y\\
0 & 0 & e^{w}&z \\
0&0&0&1
\end{matrix} \right]. \]
\noindent Right-invariant vector fields: $-\frac{1}{2}D_x,zD_x+D_y,D_z-yD_x,D_w+2xD_x+(y+z)D_y+zD_z$.\\

\noindent $A_{4.8},A_{4.9b}\,(-1 \leq b \leq 1)$ $[e_2,e_3]=e_1, [e_1,e_4]=(b+1)e_1, [e_2,e_4]=e_2, [e_3,e_4]=be_3$:
\[ S=
 \left[ \begin {array}{ccc} {{\rm e}^{ \left( b+1 \right) w}}&y{
{\rm e}^{bw}}&x\\ \noalign{\medskip}0&{{\rm e}^{bw}}&z
\\ \noalign{\medskip}0&0&1\end {array} \right]. \]
\noindent Right-invariant vector fields: $D_x, zD_x+D_y, -D_z, D_w+(xb+x)D_x+yD_y+bzD_z$.\\

\noindent $A_{4.10},A_{4.11}\, ( a \geq 0)$ $[e_2,e_3]=e_1, [e_1,e_4]=2ae_1, [e_2,e_4]=ae_2-e_3, [e_3,e_4]=e_2+ae_3$:
\[ S=\left[ \begin{matrix}
 e^{2aw} &-e^{aw}(x\sin w +y\cos w )&e^{aw}(x\cos w -y\sin w)&z\\
0 & e^{aw}\cos w& e^{aw}\sin w & x \\
0&-e^{aw}\sin w&e^{aw}\cos w&y\\
 0&0&0&1
\end{matrix} \right]. \]
\noindent Right-invariant vector fields: $-2D_z, D_x+yD_z, D_y-xD_z, D_w+(ax+y)D_x+(ay-x)D_y+2azD_z$.\\

\noindent $A_{4.12}\, [e_1,e_3]=e_1, [e_2,e_3]=e_2, [e_1,e_4]=-e_2, [e_2,e_4]=e_1$:
\[ S=\left[ \begin{matrix}
e^{z}\cos w &e^{z}\sin w & x\\
 -e^{z}\sin w& e^{z}\cos w & y\\
0&0&1
\end{matrix} \right]. \]
\noindent Right-invariant vector fields: $D_x, D_y, D_z+xD_x+yD_y,D_w+yD_x-xD_y$.\\

\noindent $A_{5.1}\,  [e_3,e_5]=e_1, [e_4,e_5]=e_2$: \[ S=\left[
\begin{matrix} 1 &q& x & z \\ 0 & 1&w& y\\ 0 & 0 & 1 &0 \\
0&0&0&1\\\end{matrix} \right].\] \noindent Right-invariant vector
fields: $-D_q, -D_x, D_y, D_z, D_w-yD_q-zD_x.$\\

\noindent $A_{5.2}\, [e_2,e_5]=e_1, [e_3,e_5]=e_2, [e_4,e_5]=e_3$: \[
S=\left[ \begin{matrix} 1 &w& \frac{w^2}{2} &  \frac{w^3}{6}&q \\ 0
& 1 &w& \frac{w^2}{2}& x\\ 0 & 0 & 1  &w&y \\ 0&0&0&1&z\\ 0&0&0&0&1
\end{matrix} \right]. \] \noindent Right-invariant vector fields:
$D_q,D_x,D_y,D_z,D_w+xD_q+yD_x+zD_y$.\\

\noindent $A_{5.3}\, [e_3,e_4] = e_2$ , $[e_3,e_5] = e_1$, $[e_4,e_5]
=e_3$: \[ S=\left[ \begin{matrix} 1 & 0 & -z&y-zw&q \\ 0 & 1
&w&\frac{w^2}{2}&x \\ 0 & 0 & 1 & w &2y \\ 0 & 0 & 0 &1&2z \\
0&0&0&0&1 \end{matrix} \right].\] \noindent Right-invariant vector
fields: $2D_x,-4D_q,D_y+2zD_q,D_z-2yD_q,D_w+2yD_x+zD_y$.\\

\noindent $A_{5.4}\, [e_2,e_4]=e_1, [e_3,e_5]=e_1$:  \[ S=\left[
\begin{matrix} 1 & x & y&q \\ 0 & 1 &0& z\\ 0 & 0 & 1  &w \\ 0&0&0&1
\end{matrix} \right]. \]
\noindent Right-invariant vector fields:
$D_q, D_z, D_w, D_x+zD_q, D_y+wD_q$.\\

\noindent $A_{5.5}\, [e_3,e_4]=e_1, [e_2,e_5]=e_1$ $[e_3,e_5]=e_2$: \[
S=\left[ \begin{matrix} 1&q &w+\frac{q^2}{2}&x \\ 0 & 1 &q&y \\ 0 &
0 & 1 &z \\ 0 & 0 & 0 &1 \\ \end{matrix} \right].\] \noindent
Right-invariant vector fields: $D_x, D_y, D_z,
D_w+zD_x,D_q+yD_x+zD_y.$\\

\noindent $A_{5.6}\,  [e_3,e_4]=e_1, [e_2,e_5]=e_1$, $[e_3,e_5]=e_2,
[e_4,e_5]=e_3$:  \[ S=\left[ \begin{matrix} 1 & 2w &
w^2-z&y-zw+\frac{w^3}{3}&q \\ 0 & 1 &w&\frac{w^2}{2}&x \\ 0 & 0 & 1
& w &y \\ 0 & 0 & 0 &1&z \\ 0&0&0&0&1 \end{matrix} \right]. \]
\noindent Right-invariant vector fields: $2D_q, -D_x, D_y+zD_q,
-D_z+yD_q,-(D_w+2xD_q+yD_x+zD_y)$.\\

\noindent $A_{5.7abc}\, (abc \neq 0, -1 \leq c \leq b \leq a \leq 1),\,
[e_1,e_5]=e_1, [e_2,e_5]=ae_2, [e_3,e_5]=be_3, [e_4,e_5]=ce_4$:
\[S=\left[ \begin{matrix} e^{w} &0& 0 &0&q \\ 0&e^{aw} &0& 0 &x \\ 0
&0& e^{bw} &0& y\\ 0 & 0&0 & e^{cw} &z \\ 0&0&0&0&1 \end{matrix}
\right]. \]
\noindent Right-invariant vector fields:
$D_q,D_x,D_y,D_z,D_w+qD_q+axD_x+byD_y+czD_z$.

$$S= \left[ \begin {array}{cccc} {{\rm e}^{aq}}&0&w{{\rm e}^{ \left(
a-1 \right) q}}&x\\ \noalign{\medskip}0&{{\rm e}^{bq}}&y{{\rm e}^{
 \left( a-1 \right) q}}&z\\ \noalign{\medskip}0&0&{{\rm e}^{ \left( a-
1 \right) q}}&0\\ \noalign{\medskip}0&0&0&1\end {array} \right] (
c=b-a+1).$$ \noindent Right-invariant vector fields: $D_w, D_x, D_z,
D_y, D_q+axD_x+(b+1-a)yD_y+bzD_z+wD_w$.\\

\noindent $A_{5.8c}\, \quad (0 < c \leq 1) \quad [e_2,e_5]=e_1$,
$[e_3,e_5]=e_3, [e_4,e_5]=ce_4 $: \[ S=\left[ \begin{matrix} e^{cw}
& 0 &0& 0&q \\ 0 & e^{w} &0&0&x \\ 0 & 0 & 1 & w &y \\ 0 & 0 & 0
&1&z \\ 0&0&0&0&1 \end{matrix} \right]. \] Right-invariant vector
fields: $D_x,D_y,D_z,D_q,D_w+cqD_q+yD_x+zD_z$.\\

\noindent $A_{5.9bc}\, \quad (0 \neq c \leq b)$ $[e_1,e_5]=e_1$,
$[e_2,e_5]=e_1+e_2, [e_3,e_5]=be_3, [e_4,e_5]=ce_4$.
\[ S=\left[
\begin{matrix} e^{cw} & 0 &0& 0&q \\ 0 & e^{bw} &0&0&x \\ 0 & 0 &
e^{w} & we^{w} &y \\ 0 & 0 & 0 &e^{w}&z \\ 0&0&0&0&1 \end{matrix}
\right].\] Right-invariant vector fields: $ D_y, D_z, D_x, D_q,
D_w+bxD_x+cqD_q+(y+z)D_y+zD_z$.\\

\noindent $A_{5.10}\, [e_2,e_5]=e_1$, $[e_3,e_5]=e_2, [e_4,e_5]=e_4$:
 \[ S=\left[ \begin{matrix}
 e^{w} & 0 &0& 0&q \\
  0 & 1 &w&\frac{w^2}{2}&x \\
  0 & 0 & 1 & w &y \\
  0 & 0 & 0 &1&z \\
  0&0&0&0&1
 \end{matrix} \right]. \]
Right invariant vector fields: $D_x,D_y,D_z,D_q ,
D_w+qD_q+yD_x+zD_y$.\\

\noindent $A_{5.11c}\, (c \neq 0)$ $[e_1,e_5]=e_1,
[e_2,e_5]=e_1+e_2,[e_3,e_5]=e_2+e_3,[e_4,e_5]=ce_4$:

\noindent \[ S=\left[ \begin{matrix} e^{cw}& 0& 0& 0& q\\ 0& e^{w}&
we^{w}& \frac{w^2e^{w}}{2}& x\\ 0& 0& e^{w}& we^{w}& y\\ 0& 0& 0&
e^{w}& z\\ 0& 0& 0& 0& 1 \end{matrix} \right]. \]
Right-invariant vector fields: $D_x, D_y, D_z, D_q,
D_w+(x+y)D_x+(y+z)D_y+zD_z+cqD_q$.\\
$$S= \left[
\begin {array}{cccc} {{\rm e}^{q}}&q{{\rm e}^{q}}&w&x \\
\noalign{\medskip}0&{{\rm e}^{q}}&y&z\\ \noalign{\medskip}0&0&1&q \\
\noalign{\medskip}0&0&0&1\end {array} \right] (c=1).$$
\noindent Right invariant vector fields: $D_x, \frac{1}{2}(D_z-D_w
-qD_x), -\frac{1}{2}(D_y+qD_z), \frac{1}{2}(D_w+D_z+qD_x),
D_q+(z+x)D_x+yD_y+zD_z+(y+w)D_w$.\\

\noindent $A_{5.12}\,$ $[e_1,e_5]=e_1, [e_2,e_5]=e_1+e_2,
[e_3,e_5]=e_2+e_3, [e_4,e_5]=e_3+e_4$. \[ S=\left[ \begin{matrix}
e^{w}& we^{w}& \frac{w^2e^{w}}{2}&\frac{w^3e^{w}}{6}&q \\ 0& e^{w}&
we^{w}& \frac{w^2e^{w}}{2}& x\\ 0& 0& e^{w}& we^{w}& y\\ 0& 0& 0&
e^{w}& z\\ 0& 0& 0& 0& 1 \end{matrix} \right].\] \noindent
Right-invariant vector fields:
$D_q,D_x,D_y,D_z,D_w+(q+x)D_q+(x+y)D_x+(y+z)D_y+zD_z+qD_q$.\\

\noindent $A_{5.13apr}\, (r \neq 0, 0 < |a| \leq 1)$ $[e_1,e_5]=e_1,
[e_2,e_5]=ae_2, [e_3,e_5]=pe_3-re_4, [e_4,e_5]=re_3+pe_4$:
\[S=\left[ \begin{matrix} e^{w}&0&0&0&q\\ 0&e^{aw}&0&0&x\\
0&0&e^{pw}\cos rw& e^{pw}\sin rw&y\\ 0&0&-e^{pw}\sin rw& e^{pw}\cos
rw&z \\ 0&0&0&0&1 \end{matrix} \right].\]
\noindent Right-invariant vector fields: $D_q,D_x, D_y,
D_z,D_w+qD_q+axD_x+(py+rz)D_y+(pz-ry)D_z$.
$$S=\left[ \begin {array}{cccc} {{\rm e}^{-q}}\cos  \frac{rq}{2}
  &{{\rm e}^{-q}}\sin  \frac{rq}{2}  &w&x
\\ \noalign{\medskip}-{{\rm e}^{-q}}\sin  \frac{rq}{2} &{
{\rm e}^{-q}}\cos \frac{rq}{2}  &y&z\\ \noalign{\medskip}0&0&
\cos \frac{rq}{2}  &\sin \frac{rq}{2}
\\ \noalign{\medskip}0&0&-\sin \frac{rq}{2} &\cos \frac{rq}{2} \end {array} \right] (a=p=1).$$
\noindent Right-invariant vector fields:
$\frac{1}{2}(\cos \frac{rq}{2}D_x-\cos \frac{rq}{2}D_y-\sin \frac{rq}{2}D_z-\sin \frac{rq}{2} D_w),
\frac{1}{2}((\cos \frac{rq}{2}-\sin \frac{rq}{2})D_w+ (\cos \frac{rq}{2}+\sin \frac{rq}{2})D_x-(\cos \frac{rq}{2}+\sin \frac{rq}{2})D_y +(\cos \frac{rq}{2}-\sin \frac{rq}{2})D_z)),
\frac{1}{2}(\cos \frac{rq}{2}D_w+\sin \frac{rq}{2}D_x+\sin \frac{rq}{2}D_y-\cos \frac{rq}{2}D_z),
\frac{1}{2}(\sin \frac{rq}{2}D_w-\cos \frac{rq}{2}D_y-\sin \frac{rq}{2}D_z-\cos \frac{rq}{2}D_x),
(x-\frac{rz}{2})D_x+(\frac{rw}{2}+y)D_y+(\frac{rx}{2}+z)D_z+(w-\frac{ry}{2})D_w-D_q$.\\

\noindent $A_{5.14p}\,[e_2,e_5]=e_1$, $[e_3,e_5]=pe_3-e_4,
[e_4,e_5]=e_3+pe_4$: \[ S=\left[ \begin{matrix} 1 & x &0& 0&q \\ 0 &
1 &0&0&w \\ 0 & 0 & e^{pw}\cos w & e^{pw}\sin w &y \\ 0 & 0
&-e^{pw}\sin w&e^{pw}\cos w &z \\ 0&0&0&0&1 \end{matrix} \right]. \]
\noindent Right-invariant vector fields: $D_q, D_x, D_y, D_z,
D_w+xD_q+(py+z)D_y+(pz-y)D_z$.\\

\noindent $A_{5.15a}\, [e_1,e_5]=e_1$,
$[e_2,e_5]=e_1+e_2,[e_3,e_5]=ae_3,[e_4,e_5]=e_3+ae_4$: $$S= \left[
\begin {array}{cccc} {{\rm e}^{aq}}&q{{\rm e}^{aq}}&w{{\rm e}^ {
\left( a-1 \right) q}}&x\\ \noalign{\medskip}0&{{\rm e}^{aq}}&y{
{\rm e}^{ \left( a-1 \right) q}}&z\\ \noalign{\medskip}0&0&{{\rm
e}^{ \left( a-1 \right) q}}&0\\ \noalign{\medskip}0&0&0&1\end
{array} \right].$$
\noindent Right-invariant vector fields: $D_w,
D_y, D_x, D_z, D_q+(z+ax)D_x+yD_y+azD_z+(w+y)D_w$.\\

\noindent $A_{5.16pr}\,( r > 0)$ $[e_1,e_5]=e_1, [e_2,e_5]=e_1+e_2,
[e_3,e_5]=pe_3-re_4,  [e_4,e_5]=re_3+pe_4$: \[ S=\left[
\begin{matrix} e^{w}& we^{w}& 0& 0& q\\ 0& e^{w}& 0& 0& x\\ 0& 0&
e^{pw}\cos rw& e^{pw}\sin rw& y\\ 0& 0& -e^{pw}\sin rw& e^{pw}\cos
rw& z\\ 0&0&0&0& 1 \end{matrix} \right]. \] Right-invariant vector
fields: $D_q,D_x,D_y,D_z,D_w+(q+x)D_q+xD_x+(py+rz)D_y+(pz-ry)D_z$.\\

\noindent $A_{5.17prs}\, (s > 0)$  $[e_1,e_5]=pe_1-e_2,
[e_2,e_5]=e_1+pe_2, [e_3,e_5]=re_3-se_4, [e_4,e_5]=se_3+re_4$: \[
S=\left[ \begin{matrix} e^{pw}\cos w& e^{pw}\sin w& 0& 0& x\\
-e^{pw}\sin w& e^{pw}\cos w& 0& 0& y\\ 0& 0& e^{rw}\cos sw&
e^{rw}\sin sw& z\\ 0& 0& -e^{rw}\sin sw& e^{rw}\cos sw& q\\
0&0&0&0&1 \end{matrix} \right]. \] Right-invariant vector fields:
$D_x,D_y,D_z,D_q,D_w+(px+y)D_x+(py-x)D_y+(rz+sq)D_z+(rq-sz)D_q$.


$$S= \left[ \begin {array}{cccc} {{\rm e}^{pq}}\cos \frac{(s+1)q}{2}
&{{\rm e}^{pq}}\sin  \frac{(s+1)q}{2} &w&x\\
\noalign{\medskip}-{{\rm e}^{pq}}\sin
 \frac{(s+1)q}{2}  &{{\rm e}^{pq}}\cos \frac{(s+1)q}{2}
&y&z\\ \noalign{\medskip}0 &0&\cos  \frac{(s-1)q}{2} & \sin
\frac{(s-1)q}{2} \\ \noalign{\medskip}0&0&-\sin
 \frac{(s-1)q}{2} &\cos  \frac{(s-1)q}{2} \end {array} \right]\, (r=p).$$

\noindent Right-invariant vector fields: $\cos \frac{(s-1)q}{2} D_z-\sin
\frac{(s-1)q}{2}D_y-\sin \frac{(s-1)q}{2}D_x-\cos
\frac{(s-1)q}{2}D_w, \cos
\frac{(s-1)q}{2}D_y+\cos \frac{(s-1)q}{2}D_x-\sin
\frac{(s-1)q}{2}D_w+\sin \frac{(s-1)q}{2}D_z, \sin \frac{(s-1)q}{2}D_x-\sin
\frac{(s-1)q}{2}D_y+\cos
\frac{(s-1)q}{2}D_w+\cos \frac{(s-1)q}{2}D_z,\\ -\cos
\frac{(s-1)q}{2}D_y+\cos \frac{(s-1)q}{2}D_x-\sin
\frac{(s-1)q}{2}D_w-\sin \frac{(s-1)q}{2}D_z,
 D_q+\frac{(s+1)z+2py}{2}D_y-\frac{(s+1)z-2px}{2}
D_x-\frac{(s+1)y-2pw}{2}D_w+\frac{(s+1)x+2pz}{2}D_z$.\\
$$S= \left[
\begin {array}{cccc} \cos q  &\sin q  &w&x\\ \noalign{\medskip}-\sin
q  &\cos q  &y&z\\ \noalign{\medskip}0&0&{{\rm e}^{pq}}&0 \\
\noalign{\medskip}0&0&0&{{\rm e}^{rq}}\end {array} \right]\,
(s=1).$$ \noindent Right-invariant vector fields: $e^{pq}D_w,
e^{pq}D_y, e^{rq}D_x, e^{rq}, -D_q-zD_x+wD_y+xD_z-yD_w$.\\

\noindent $A_{5.18p}\, ( p \geq 0$)  $[e_1,e_5]=pe_1-e_2,
[e_2,e_5]=e_1+pe_2,  [e_3,e_5]=e_1+pe_3-e_4, [e_4,e_5]=e_2+e_3+pe_4$
$$S=  \left[ \begin {array}{cccc} \cos q &\sin  q  &w&x\\
\noalign{\medskip}-\sin  q &\cos  q  &y&z\\
\noalign{\medskip}0&0&{{\rm e}^{pq}}&q{{\rm e}^{pq}} \\
\noalign{\medskip}0&0&0&{{\rm e}^{pq}}\end {array} \right] .$$

\noindent Right-invariant vector fields: $e^{pq}D_x,-e^{pq}D_z,
e^{pq}(D_w + q D_x),-e^{pq}(D_y+qD_z),-D_q-z D_x+w D_y+x D_z-y D_w$.\\

\noindent $A_{5.19ab}\,(b \neq 0)\,, [e_2,e_3]=e_1, [e_1,e_5]=ae_1,
[e_2,e_5]=e_2, [e_3,e_5]=(a-1)e_3, [e_4,e_5]=be_4$: \[S=\left[
\begin {array}{cccc} {e^{aw}}&{e^{w}}x&0&z\\\noalign{\medskip}0
&{e^{w}}&0&y\\\noalign{\medskip}0&0&{e^{bw}}&q\\\noalign{\medskip}0&0&0
&1\end {array} \right].\] \noindent Right-invariant vector fields:
$D_z, D_y, D_x+yD_z, D_q, D_w+bqD_q+(a-1)xD_x+yD_y+azD_z$.\\

\noindent $A_{5.20a}\, [e_2,e_3]=e_1$, $[e_1,e_5]=ae_1$,
$[e_2,e_5]=e_2$, $[e_3,e_5]=(a-1)e_3$, $[e_4,e_5]=e_1+ae_4$:

\[S=\left[ \begin {array}{cccc} {e^{aw}}&{e^{w}}x&w{e^{aw}}&z
\\\noalign{\medskip}0&{e^{w}}&0&y\\\noalign{\medskip}0&0&{e^{aw}}&q
\\\noalign{\medskip}0&0&0&1\end {array} \right].\] \noindent
Right-invariant vector fields:  $D_z, D_y, D_x + y D_z, D_q, D_w+
aqD_q + (a-1)xD_x + y D_y + (q + a z) D_z$.\\

\noindent $A_{5.21}\, [e_2,e_3]=e_1,[e_1,e_5]=2e_1, [e_2,e_5]=e_2+e_3,
[e_3,e_5]=e_3+e_4, [e_4,e_5]=e_4$: \[ S=\left[ \begin{matrix}
e^{2w}&0&ze^{w}&(z-y+zw)e^{w}&q \\
0&e^{w}&we^{w}&\frac{w^2e^{w}}{2}&x\\ 0&0&e^{w}&we^{w}&y\\ 0 & 0 & 0
&e^{w}&z \\ 0&0&0&0&1 \end{matrix} \right]. \] Right-invariant
vector fields: $-2D_q,(y+z)D_q+D_z,D_y-zD_q, D_x,
D_w+2qD_q+(x+y)D_x+(y+z)D_y+zD_z.$\\

\noindent $A_{5.22}\, [e_2,e_3]=e_1$, $[e_2,e_5]=e_3$, $[e_4,e_5]=e_4$:
\[ S=\left[ \begin{matrix} e^{w} & 0&0& 0&q \\ 0 & 1
&z&\frac{z^2}{2}&x \\ 0 & 0 & 1 & z &y \\ 0 & 0 & 0 &1&w \\
0&0&0&0&1 \end{matrix} \right]. \] Right-invariant vector fields:
$D_x, D_z+yD_x+wD_y, -D_y, D_q, D_w+qD_q.$\\

\noindent $A_{5.23b}\, (b\neq 0) [e_2,e_3]=e_1, [e_1,e_5]=2e_1,
[e_2,e_5]=e_2+e_3,[e_3,e_5]=e_3,[e_4,e_5]=be_4$: \[S
=\left[\begin{matrix} e^{bw}& 0& 0& 0&q\\
 0& e^{2w}& -ze^{w}& ye^{w}& x\\
0& 0& e^{w}& we^{w}& y+zw\\ 0& 0&0& e^{w}& z\\ 0& 0& 0& 0& 1
\end{matrix} \right]. \] Right-invariant vector fields:
$-\frac{1}{2}D_x,zD_x+D_y,D_z-(y+zw)D_x-wD_y,D_q,D_w+qD_q+2xD_x+yD_y+zD_z.$

 \[S = \left[ \begin {matrix} 1&y{{\rm e}^{w}}&
\left( q-yw-2\,z
 \right) {{\rm e}^{w}}&x{{\rm e}^{2\,w}}\\ \noalign{\medskip}0&{
{\rm e}^{w}}&-w{{\rm e}^{w}}&q{{\rm e}^{2\,w}}\\
\noalign{\medskip}0&0 &{{\rm e}^{w}}&y{{\rm e}^{2\,w}}\\
\noalign{\medskip}0&0&0&{{\rm e}^{2 \,w}}\end {matrix} \right]
(b=1).\]

\noindent Right-invariant vector fields: $-2D_x, D_y+qD_x,
D_z+D_q-yD_x, \frac{1}{2}(D_q+yD_x),
-D_w+(q+y)D_q+2xD_x+yD_y+(y+z)D_z$.\\

\noindent $A_{5.24\epsilon}\,(\epsilon = \pm 1)$ $[e_2,e_3]=e_1,
[e_1,e_5]=2e_1, [e_2,e_5]=e_2+e_3, [e_3,e_5]=e_3, [e_4,e_5]=\epsilon
e_1+2e_4$:
$$S= \left[ \begin {array}{cccc} 1&q&z{{\rm e}^{q}}&x{{\rm
e}^{2\,q}} \\ \noalign{\medskip}0&1&y{{\rm e}^{q}}&
\left(\frac{y^2}{2}-\epsilon w
 \right) {{\rm e}^{2\,q}}\\ \noalign{\medskip}0&0&{{\rm e}^{q}}&y{
{\rm e}^{2\,q}}\\ \noalign{\medskip}0&0&0&{{\rm e}^{2\,q}}\end
{array} \right].$$ \noindent Right-invariant vector fields:
$D_x,D_y,-D_z-yD_x,D_w,-D_q+(2x-\frac{y^2}{2}+\epsilon
w)D_x+yD_y+(z-y)D_z+2wD_w.$\\

\noindent $A_{5.25bp}\, (b \neq 0$) $[e_2,e_3]=e_1$, $[e_1,e_5]=2pe_1$,
$[e_2,e_5]=pe_2+e_3$, $[e_3,e_5]=pe_3-e_2$, $[e_4,e_5]=be_4$:

$$S= \left[ \begin {array}{ccccc} \noalign{\medskip} {{\rm
e}^{2\,pw}}&-{{\rm e}^{pw}}(y\cos w+x\sin w) & {{\rm e}^{pw}}(x\cos
w -y\sin w)&0&z\\ \noalign{\medskip}0&{{\rm e}^{pw}}\cos w & {{\rm
e} ^{pw}}\sin w & 0 &x\\ \noalign{\medskip}0&-{{\rm e}^{pw}} \sin
 w  & {{\rm e}^{pw}}\cos w & 0 & y\\
0&0&0&e^{bw}&q\\ \noalign{\medskip}0&0&0&0&1\end {array} \right].$$

\noindent Right-invariant vector fields:  $2D_z, D_x+yD_z,
-D_y+xD_z, D_q, D_w+bqD_q+(px+y)D_x+(py-x)D_y+2pzD_z$.\\

\noindent  $A_{5.26\epsilon p} (\epsilon = \pm 1)\, [e_2,e_3]=e_1,
[e_1,e_5]=2pe_1, [e_2,e_5]=pe_2+e_3,
  [e_3,e_5]=pe_3-e_2, [e_4,e_5]=\epsilon e_1+2pe_4$:

$$S=\left[ \begin {array}{ccccc} {{\rm e}^{2\,pw}}&-{{\rm e}^{pw}}
 \left( y\cos  w  +x\sin  w   \right) &{
{\rm e}^{pw}} \left( x\cos w  -y\sin  w
 \right) &2\,\epsilon\,w{{\rm e}^{2\,pw}}&z\\ \noalign{\medskip}0&{
{\rm e}^{pw}}\cos  w  &{{\rm e}^{pw}}\sin w
 &0&x\\ \noalign{\medskip}0&-{{\rm e}^{pw}}\sin w
&{{\rm e}^{pw}}\cos  w  &0&y \\ \noalign{\medskip}0&0&0&{{\rm
e}^{2\,pw}}&q\\ \noalign{\medskip}0&0 &0&0&1\end {array} \right].$$

\noindent Right-invariant vector fields: $2D_z, D_x+yD_z, -D_y+xD_z,
D_q, D_w+(px+y)D_x+(py-x)D_y+2(\epsilon q+pz)D_z+2pqD_q.$\\

\noindent $A_{5.27} \, [e_2,e_3]=e_1, [e_1,e_5]=e_1, [e_3,e_5]=e_3+e_4,
[e_4,e_5]=e_1+e_4$:
\[S= \left[ \begin {array}{cccc} {{\rm e}^{w}}&w{{\rm e}^{w}}& \left( q+
\frac{w^2}{2} \right) {{\rm e}^{w}}&z\\ \noalign{\medskip}0&{{\rm e}^{w}}
&w{{\rm e}^{w}}&x\\ \noalign{\medskip}0&0&{{\rm e}^{w}}&y
\\ \noalign{\medskip}0&0&0&1\end {array} \right].\]
\noindent Right-invariant vector fields: $D_z, -D_q-yD_z, D_y, D_x, D_w+(x+y)D_x+yD_y+(x+z)D_z.$\\


\noindent $A_{5.28a}\, [e_2,e_3]=e_1$, $[e_1,e_5]=ae_1$,
$[e_2,e_5]=(a-1)e_2$, $[e_3,e_5]=e_3+e_4$, $[e_4,e_5]=e_4$:
$$S=\left[ \begin {array}{cccc} {{\rm e}^{-w}}&z&x{{\rm e}^{- \left(
a-1 \right) w}}&q\\ \noalign{\medskip}0&1&y{{\rm e}^{ \left( a-1
\right) w}}&w\\ \noalign{\medskip}0&0&{{\rm e}^{ \left( a-1 \right)
w}}&0 \\ \noalign{\medskip}0&0&0&1\end {array} \right].$$
\noindent Right-invariant vector fields: $e^{2(a-1)w}D_x, D_y,
ye^{(2(a-1)w}D_x+D_z+wD_q, D_q,
-D_w-(a-2)xD_x+(a-1)yD_y+zD_z+qD_q$.\\

\noindent $A_{5.29}\,[e_2,e_4]=e_1$, $[e_1,e_5]=e_1$, $[e_2,e_5]=e_2$,
$[e_4,e_5]=e_3$:
$$S= \left[ \begin {array}{cccc} 1&q&x{{\rm e}^{w}}&z \\
\noalign{\medskip}0&1&y{{\rm e}^{w}}&w\\ \noalign{\medskip}0&0&{
{\rm e}^{w}}&0\\ \noalign{\medskip}0&0&0&1\end {array} \right].$$
\noindent Right-invariant vector fields: $D_x, D_y, D_z,
D_q+yD_x+wD_z, -D_w+xD_x+yD_y.$\\

\noindent $A_{5.30a} \,  [e_2,e_4]=e_1, [e_3,e_4]=e_2$,
$[e_1,e_5]=(a+1)e_1$, $[e_2,e_5]=ae_2$, $[e_3,e_5]=(a-1)e_3$,
$[e_4,e_5]=e_4$:
$$S=\left[ \begin {array}{cccc}
1&q&\frac{q^2}{2}&x\\\noalign{\medskip}0&{
e^{w}}&{e^{w}}q&{e^{w}}y\\\noalign{\medskip}0&0&{e^{2\,w}}&{e^{2\,w}}z
\\\noalign{\medskip}0&0&0&{e^{ \left( a+1 \right) w}}\end
{array}\right].$$

\noindent Right-invariant vector fields: $e^{(a+1)w}D_x, e^{aw}D_y,
e^{(a-1)w}D_z, e^w(D_q+yD_x+zD_y), -D_w$.\\

\noindent $A_{5.31}\, [e_2,e_4]=e_1$, $[e_3,e_4]=e_2$, $[e_1,e_5]=3e_1$,
$[e_2,e_5]=2e_2$, $[e_3,e_5]=e_3$,$[e_4,e_5]=e_3+e_4$:  \[ S=\left[ \begin{matrix} e^{3w}& -ze^{2w}&
\frac{1}{2}z^2e^{w}&\frac{1}{2}e^{w}(z^2w+x-yz+\frac{3z^2}{2})& q\\
0& e^{2w}& -ze^{w}&e^{w}(y-z-zw)& x\\ 0& 0& e^{w}& we^{w}&y\\
0&0&0&e^{w}&z\\ 0&0&0&0&1 \end{matrix} \right]. \]
\noindent Right-invariant vector fields: $3D_q, -(2D_x+z D_q), D_y+zD_x, D_z -xD_q-(y+z)D_x,
D_w+3q D_q + 2xD_x+(y + z) D_y+ + zD_z$.\\

\noindent $A_{5.32a}\, [e_2,e_4]=e_1$, $[e_3,e_4]=e_2$, $[e_1,e_5]=e_1$,
$[e_2,e_5]=e_2$, $[e_3,e_5]=ae_1+e_3$:
$$S= \left[ \begin {array}{cccc} 1&q&\frac{q^2}{2}-aw&x \\
\noalign{\medskip}0&1&q&y\\ \noalign{\medskip}0&0&1&z \\
\noalign{\medskip}0&0&0&{{\rm e}^{w}}\end {array} \right].$$
\noindent Right-invariant vector fields: $e^w D_x, e^w D_y, e^w D_z,
D_q+yD_x+zD_y, -D_w+azD_x$.\\

\noindent $A_{5.33ab}\, (a^2+b^2) \neq 0$ $[e_1,e_4]=e_1$,
$[e_3,e_4]=be_3$, $[e_2,e_5]=e_2$, $[e_3,e_5]=ae_3$:
\[ S= \left[
\begin {array}{cccc} {e^{z}}&0&0&q\\\noalign{\medskip}0&{e^{w}
}&0&x\\\noalign{\medskip}0&0&e^{aw+bz}&y\\\noalign{\medskip}0&0&0&1
\end {array} \right].\]
\noindent Right-invariant vector fields:
$D_q, D_x, D_y, -(D_z+qD_q+byD_y), -(D_w+xD_x+ayD_y).$\\

\noindent $A_{5.34a}\, [e_1,e_4]=a
e_1,[e_2,e_4]=e_2,[e_3,e_4]=e_3,[e_1,e_5]=e_1,[e_3,e_5]=e_2$.
$$S=
\left[ \begin {array}{cccc} {{\rm e}^{\alpha z+w}}&0&0&q
\\\noalign{\medskip}0&{{\rm e}^{z}}&w&x\\\noalign{\medskip}0&0&{
{\rm e}^{z}}&y\\\noalign{\medskip}0&0&0&1\end {array} \right].$$
\noindent Right-invariant vector fields: $D_q, D_x, D_y,
D_z+xD_x+yD_y+aqD_q, D_w+yD_x+qD_q$.\\

\noindent $A_{5.35ab} \, (a^2+b^2 \neq 0)$ $[e_1,e_4]=b
e_1,[e_2,e_4]=e_2,[e_3,e_4]=e_3 ,[e_1,e_5]=a
e_1,[e_2,e_5]=-e_3,[e_3,e_5]=e_2$: $$S= \left[ \begin {array}{cccc}
{{\rm e}^{aw+bz}}&0&0&q \\ \noalign{\medskip}0&{{\rm e}^{z}}\cos w
&{{\rm e}^{z }}\sin w &x\\ \noalign{\medskip}0&-{{\rm e}^{z}}\sin w
&{{\rm e}^{z}}\cos w &y \\ \noalign{\medskip}0&0&0&1\end {array}
\right].$$\\
\noindent Right-invariant vector fields: $D_q, D_x, D_y,
D_z+xD_x+yD_y+bqD_q, D_w+yD_x-xD_y+aqD_q$.\\

\noindent $A_{5.36} \, [e_2,e_3]=e_1, [e_1,e_4]=e_1,
[e_2,e_4]=e_2, [e_2,e_5]=-e_2, [e_3,e_5]=e_3$:
$$S=\left[ \begin {array}{ccc} {{\rm e}^{w}}&{{\rm e}^{q}}x&z \\
\noalign{\medskip}0&{{\rm e}^{q}}&y\\ \noalign{\medskip}0&0&1 \end
{array} \right].$$
\noindent Right-invariant vector fields: $D_z,-D_x-yD_z, D_y,
Dw+xD_x+zD_z, D_q-xD_x+yD_y$.\\

\noindent $A_{5.37}\, [e_2,e_3]=e_1$, $[e_1,e_4]=2e_1$,$[e_2,e_4]=e_2$,
$[e_3,e_4]=e_3$, $[e_2,e_5]=-e_3$,$[e_3,e_5]=e_2$:
$$S= \left[ \begin {array}{cccc} {{\rm e}^{2\,q}}& \left( y\cos w
+x\sin w  \right) {{\rm e}^{q}}& \left( y\sin
 w  -x\cos  w   \right) {{\rm e}^{q}}&z
\\ \noalign{\medskip}0&{{\rm e}^{q}}\cos w & {{\rm e}^{q}} \sin
 w &x\\ \noalign{\medskip}0&-{{\rm e}^{q}}\sin w
 & {{\rm e}^{q}} \cos  w &y
\\ \noalign{\medskip}0&0&0&1\end {array} \right].$$
\noindent Right-invariant vector fields: $2D_z, D_x-yD_z, D_y+xD_z,
D_q+xD_x+yD_y+2zD_z, D_w+yD_x-xD_y$.\\

\noindent $A_{5.38}\, [e_1,e_4]=e_1$, $[e_2,e_5]=e_2$, $[e_4,e_5]=e_3$:
$$ S=\left[ \begin{matrix} e^{z} & 0&0& 0&q \\ 0 & e^{w} &0& 0&x \\
0 & 0 & 1 & w &y \\ 0 & 0 & 0 &1&z \\ 0&0&0&0&1 \end{matrix}
\right].$$  Right-invariant vector fields: $D_q, D_x, D_y, qD_q+D_z,
D_w+xD_x+zD_y$.\\

\noindent $A_{5.39}\, [e_1,e_4]=e_1$, $[e_2,e_4]=e_2$, $[e_1,e_5]=-e_2$,
 $[e_2,e_5]=e_1$, $[e_4,e_5]=e_3$:

\[ S= \left[ \begin{matrix} 1&z&0&0&q\\ \noalign{\medskip}0&1&0&0&w
\\ \noalign{\medskip}0&0&{{\rm e}^{-w}}\cos \left( z \right) &{{\rm e}
^{-w}}\sin \left( z \right) &x\\ \noalign{\medskip}0&0&-{{\rm e}^{-w}}
\sin \left( z \right) &{{\rm e}^{-w}}\cos \left( z \right) &y
\\ \noalign{\medskip}0&0&0&0&1\end{matrix} \right].\]
\noindent Right-invariant vector fields: $D_y, D_x, D_q, -D_w+yD_y+xD_x, -D_z+xD_y-yD_x-wD_q$.\\

\noindent $A_{5.40}\, [e_1,e_2]=2e_1, [e_3,e_1]=e_2, [e_2,e_3]=2e_3,
[e_1,e_4]=e_5, [e_2,e_4]=e_4, [e_2,e_5]=-e_5, [e_3,e_5]=e_4$:
\[S=\left[ \begin{matrix} e^{x} &y&w  \\ z & (1+yz)e^{-x}&q \\ 0&0&1
\end{matrix} \right]. \] Right-invariant vector fields:
$ze^{-x}D_x+(yz+1)e^{-x}D_y+qD_w,
-qD_q+D_x+yD_y-zD_z+wD_w,wD_q+e^{x}D_z, D_q, D_w$.



\section{Appendix}

\begin{proposition} Algebra $A_{5,21}$ has no representation in
$\g\l(4,\R)$. \end{proposition}

\begin{proof} First of all assume that there is a upper triangular representation. Since each of $e_1,e_2,e_3,e_4$ is a sum of
commutators we may assume that each of $E_1,E_2,E_3,E_4$ are strictly upper triangular. Now put

$$E_2=\left[\begin{smallmatrix}0&t&u&v\\ \noalign{\medskip}0&0&p&q\\
\noalign{\medskip}0&0&0&r\\ \noalign{\medskip}0&0&0&0\end
{smallmatrix}
 \right],
E_5=\left[ \begin{smallmatrix}a&b&c&d\\ \noalign{\medskip}0&e&f&g\\
\noalign{\medskip}0&0&h&i\\ \noalign{\medskip}0&0&0&j\end
{smallmatrix}
 \right]$$

\noindent and \emph{define} $E_3=[E_2,E_5]-E_2, E_4=[E_3,E_5]-E_3,
E_1=[E_2,E_3]$. Then

$$E_1 = \left[ \begin {smallmatrix}0&0&tp \left( h-2\,e+a \right)
&tpi+tqj-2 \,teq-2\,tfr+urj-2\,urh+tqa+rau+rbp\\
\noalign{\medskip}0&0&0&pr
 \left( j-2\,h+e \right) \\ \noalign{\medskip}0&0&0&0
\\ \noalign{\medskip}0&0&0&0\end {smallmatrix} \right]$$

$$E_3=\left[ \begin {smallmatrix} 0&t \left(e-a-1 \right)
&tf+uh-au-bp-u &tg+ui+vj-av-bq-cr-v\\ \noalign{\medskip}0&0&p
\left(h-e-1 \right) &pi+qj-eq-fr-q\\ \noalign{\medskip}0&0&0&r
\left(j-h-1 \right) \\ \noalign{\medskip}0&0&0&0\end {smallmatrix}
\right], E_4= \left[ \begin {smallmatrix} 0&t \left(a-e+1 \right)
^{2}& * 
& 
\,* \\ \noalign{\medskip}0&0&p \left( -h+e+1 \right) ^{2}&
*\\
\noalign{\medskip}0&0&0&r \left( -j+h+1 \right) ^{2} \\
\noalign{\medskip}0&0&0&0\end{smallmatrix}\right].$$

It turns out that $[E_1,E_5]=[E_2,E_4]$, the equality of
commutators. Furthermore we have that $[E_2,E_5]=E_2+E_3,
[E_3,E_5]=E_3+E_4, [E_2,E_3]=E_1$ by construction. The remaining
brackets are given by
$$[E_1,E_2]= \left[ \begin {smallmatrix} 0&0&0&tpr \left(
3\,h-3\,e+a-j \right) \\ \noalign{\medskip}0&0&0&0\\
\noalign{\medskip}0&0&0&0 \\ \noalign{\medskip}0&0&0&0\end
{smallmatrix} \right], [E_1,E_3]=\left[ \begin {smallmatrix}
0&0&0&tpr \left( jh-{h}^{2}-3\,h-3\,je+4
\,he+3\,e+2\,ja-3\,ha-a-{e}^{2}+ea+j \right) \\
\noalign{\medskip}0&0&0 &0\\ \noalign{\medskip}0&0&0&0\\
\noalign{\medskip}0&0&0&0\end {smallmatrix} \right],$$

$$[E_1,E_4]=  \left[ \begin {smallmatrix} 0&0&0&*\\
 \noalign{\medskip}0&0&0&0\\
\noalign{\medskip}0&0 &0&0\\ \noalign{\medskip}0&0&0&0\end
{smallmatrix} \right], [E_1,E_5]-2E_5= \left[ \begin {smallmatrix}
0&0&tp \left( h-a-2 \right) \left( h-2 \,e+a \right)
&*\\
\noalign{\medskip}0&0&0&pr \left( j-e-2 \right)  \left( j-2\,h+e
 \right) \\ \noalign{\medskip}0&0&0&0\\ \noalign{\medskip}0&0&0&0
\end {smallmatrix} \right],$$

$$[E_1,E_4]=  \left[ \begin {smallmatrix} 0&0&0&*\\
 \noalign{\medskip}0&0&0&0\\
\noalign{\medskip}0&0 &0&0\\ \noalign{\medskip}0&0&0&0\end
{smallmatrix} \right], [E_3,E_4]=  \left[ \begin {smallmatrix} 0&0&tp \left( h-e-1
\right) \left( e-a -1 \right)  \left( h-2\,e+a \right)
&*\\
\noalign{\medskip}0&0&0&pr \left( j-h-1 \right)  \left( h-e-1
 \right)  \left( j-2\,h+e \right) \\ \noalign{\medskip}0&0&0&0
\\ \noalign{\medskip}0&0&0&0\end {smallmatrix} \right],$$

$$[E_4,E_5]-E_4=   \left[ \begin {smallmatrix} 0&t \left( e-a-1
\right) ^{3}&*&*\\
\noalign{\medskip}0&0&p \left( h-e-1 \right)
^{3}&*\\
\noalign{\medskip}0&0&0&r \left( j-h-1
 \right) ^{3}\\ \noalign{\medskip}0&0&0&0\end {smallmatrix} \right].$$

Looking at the $[E_1,E_2]$ bracket we see that $prt(3h-3e+a-j)=0$.
Suppose first of all that $r=0$. Then $t\neq0$ or else $E_1=0$. From
the $(1,2)$-entry in the $[E_4,E_5]$ bracket we deduce that $e=a+1$.
Now it must be that $p\neq0$ or else $[E_1,E_5]=2E_5$ implies that
$E_1=0$. Now from the $(1,3)$-entry in the $[E_1,E_5]$ bracket we
deduce that $h=a+2$. From the $(1,3)$-entry in the $[E_4,E_5]$
bracket we conclude that $u=bp-ft$. However, we now have a
contradiction because $E_1$ and $E_4$ are proportional.

Hence we may assume that $r\neq0$ and by appealing to Corollary 3.2
that also $t\neq0$ and hence $rt\neq0$. Now suppose that $p=0$. Then
comparing $E_1$ and the $[E_1,E_5]$ bracket we find that
$j=a+2$. From the $(1,2)$ and $(4,5)$-entries in the $[E_4,E_5]$
bracket we find that $e=a+1, h=a+1$ which is a  contradiction
because now $E_1$ and $E_4$ are proportional.

Hence we assume that $prt\neq0$. Then from the
$(1,2), (2,3)$ and $(4,5)$-entries in the $[E_4,E_5]$ bracket we find that
$e=a+1,h=e+1,j=h+1$ which is a  contradiction because now $E_1,E_3$ and
$E_4$ are linearly dependent. Hence there can be no representation
of $A_{5,21}$ in $\g\l(4,\R)$.\end{proof}

\begin{proposition} For algebra $A_{5,22}$ we have $\mu=5$.
\end{proposition}
We assume that algebra $A_{5,22}$
has an upper triangular representation. In fact we my assume that $E_2$ and $E_5$ are upper triangular and that $E_4$ is strictly upper triangular, $E_1$ and $E_3$ being determined by the brackets that define the algebra. We shall write

$$E_2=\left[\begin{smallmatrix}\alpha&\beta&\delta&\rho\\ \noalign{\medskip}\alpha&\beta&\delta&\rho\\
\noalign{\medskip}0&\lambda&\sigma&\tau\\ \noalign{\medskip}0&0&\phi&\mu\end{smallmatrix}\right],
E_5=\left[\begin{smallmatrix}c&d&e&f\\ \noalign{\medskip}0&g&h&i\\
\noalign{\medskip}0&0&j&k\\ \noalign{\medskip}0&0&0&m\end
{smallmatrix}
\right]$$
\noindent From the $(1,2),(2,3),(3,4)$-entries of $[E_1,E_2]=0$ it follows that the $(1,2),(2,3),(3,4)$-entries of $E_1$ are zero. From $[E_2,E_3]=E_1$ and $[E_3,E_5]=0$ it follows that the $(1,2),(2,3),(3,4)$-entries of $E_3$ are zero.
Invoking Corollary 3.2 we make an upper-triangular transformation so as to
reduce $E_4$ to one of the following seven forms and work through each of these seven cases in turn:
\[ \begin{array}{lllllll} & \left[\begin{smallmatrix} 0&1&0&0\\
0&0&1&0\\ 0&0&0&1\\ 0&0&0&0 \end{smallmatrix}\right],
\left[\begin{smallmatrix} 0&1&0&0\\ 0&0&1&0\\ 0&0&0&0\\ 0&0&0&0
\end{smallmatrix}\right] ,\,\left[\begin{smallmatrix} 0&1&0&0\\
0&0&0&0\\ 0&0&0&1\\ 0&0&0&0 \end{smallmatrix}\right], \left[\begin{smallmatrix} 0&0&1&0\\
0&0&0&1\\ 0&0&0&0\\ 0&0&0&0 \end{smallmatrix}\right] , \left[\begin{smallmatrix} 0&0&1&0\\
0&0&0&0\\ 0&0&0&0\\ 0&0&0&0 \end{smallmatrix}\right] \left[\begin{smallmatrix} 0&0&0&0\\
0&0&1&0\\ 0&0&0&0\\ 0&0&0&0 \end{smallmatrix}\right]
, \left[\begin{smallmatrix} 0&0&0&1\\
0&0&0&0\\ 0&0&0&0\\ 0&0&0&0 \end{smallmatrix}\right].
\end{array} \]

In the interests of saving space we do not write out explicitly the very complicated matrices concerned.  
In the first case from $[E_2,E_4]=0$ we find that
$\phi=\alpha, \sigma=\beta, \tau=\delta, \lambda=\alpha, \mu=\beta, \nu=\alpha$ and from $[E_4,E_5]=E_4$ that
$g=c+1, h=d, i=e, j=g+1, k=h, t=j+1$. Now, however, $[E_3,E_5]=0$ implies that $E_3=0$.

In the second case from $[E_2,E_4]=0$ we find that
$\phi=\alpha, \sigma=\beta, \lambda=\alpha, \tau=0, \mu=0$ and from $[E_4,E_5]=E_4$ that
$g=c+1, h=d, i=0, j=g+1, k=0$. Now, however, $[E_1,E_2]=0$ implies that $E_1=0$.

In the third case from $[E_2,E_4]=0$ we find that
$\phi=\alpha, \sigma=0, \tau=\delta, \nu=\lambda$ and from $[E_4,E_5]=E_4$ that
$g=c+1, h=0, i=e, t=j+1$. The $(1,3)$ entries of $[E_1,E_2]=0$
$[E_3,E_5]=0$ imply that $e\alpha+j\delta-c\delta-e\nu=0$ and hence $E_1$ and $E_3$ are proportional.

In the fourth case from $[E_2,E_4]=0$ we find that
$\lambda=\alpha, \mu=\beta, \nu=\phi$ and from $[E_4,E_5]=E_4$ that
$j=c+1, k=d, t=g+1$. Now the only non-zero entry of the commutator $[E_2,E_3]$ is in the $(1,4)$ position.
On the hand the $(1,3)$ and $(2,4)$ entries of
$[E_3,E_5]$ imply that $h\beta+\delta-d\sigma=0$ and $d\sigma+\tau-h\beta=0$ and hence the only non-zero entry of $E_3$ is in the $(1,4)$ position. Thus if also  $[E_2,E_3]=E_1$ then $E_1$ and $E_3$ would be proportional.


In the fifth case from $[E_2,E_4]=0$ we find that $\lambda=\alpha, \mu=0$ and from $[E_4,E_5]=E_4$ that $j=c+1, k=0$. The $(1,3)$-entry of $[E_3,E_5]=0$ implies that the $(1,3)$-entry of $E_3$ is zero and $(2,4)$-entries of $[E_1,E_2]=0$ and $[E_3,E_5]=0$ that the $(2,4)$-entry of $E_3$ is zero. Now $[E_2,E_3]=E_1$ gives that $E_1$ and $E_3$ are proportional.

In the sixth case from $[E_2,E_4]=0$ we find that $\phi=\lambda, \beta=0, \mu=0$ and from $[E_4,E_5]=E_4$ that $d=0, j=g+1, k=0$. Now, however, $[E_1,E_2]=0$ immediately implies that $E_1=0$.

In the seventh case from $[E_2,E_4]=0$ we find that $\nu=\alpha$ and from $[E_4,E_5]=E_4$ that $t=c+1$.
Then $[E_2,E_3]=E_1$ implies that the $(1,3)$ entry of  $E_1$ is zero. If $(\alpha-\lambda)(\alpha-\phi)\neq0$ then comparing $E_1$ and $[E_1,E_2]=0$ implies that $E_1$ and $E_4$ are proportional. Hence by appealing to Corollary 3.2 we may assume that
$\lambda=\alpha$. Again comparing $E_1$ and $[E_1,E_2]$ we find that
$\phi=\alpha$ in order not to have $E_1$ and $E_4$ proportional. Now, however, $[E_2,E_3]=E_1$ gives that $E_1$ and $E_4$ are proportional.


\begin{proposition} $\mu(\g)=5$ for algebra $A_{5,23}^b$ for $b\neq 1$
and $\mu(\g)=4$ for $b=1$. \end{proposition}

We define
$$E_2= \left[ \begin {smallmatrix} 0&\beta&\delta&\rho\\ \noalign{\medskip}0
&0&\sigma&\tau\\ \noalign{\medskip}0&0&0&\mu\\ \noalign{\medskip}0&0&0
&0\end {smallmatrix} \right],
E_4= \left[ \begin{smallmatrix} 0&m&n&p\\ \noalign{\medskip}0&0&q&r
\\ \noalign{\medskip}0&0&0&s\\ \noalign{\medskip}0&0&0&0\end {smallmatrix}
 \right],
E_5= \left[ \begin {smallmatrix} c&d&e&f\\ \noalign{\medskip}0&g&h&i
\\ \noalign{\medskip}0&0&j&k\\ \noalign{\medskip}0&0&0&t\end {smallmatrix}
 \right]$$
and then $E_3$ is defined as $[E_2,E_5]-E_2$ and $E_1$ as $[E_2,E_3]$ giving

$$E_3= \left[ \begin {smallmatrix} 0&-\beta\, \left( -g+c+1 \right) &\beta\,
h+\delta\,j-c\delta-d\sigma-\delta&\beta\,i+\delta\,k+\rho\,t-c\rho-d
\tau-e\mu-\rho\\ \noalign{\medskip}0&0&-\sigma\, \left( -j+g+1
 \right) &\sigma\,k+\tau\,t-g\tau-h\mu-\tau\\ \noalign{\medskip}0&0&0&
-\mu\, \left( -t+j+1 \right) \\ \noalign{\medskip}0&0&0&0\end {smallmatrix}
 \right]$$

$$E_1= \left[ \begin {smallmatrix} 0&0&\beta\,\sigma\, \left( j-2\,g+c
 \right) &\beta\,\sigma\,k+\beta\,\tau\,t-2\,\beta\,g\tau-2\,\beta\,h
\mu+\delta\,\mu\,t-2\,\delta\,\mu\,j+\beta\,\tau\,c+\mu\,c\delta+\mu\,
d\sigma\\ \noalign{\medskip}0&0&0&\sigma\,\mu\, \left( t-2\,j+g
 \right) \\ \noalign{\medskip}0&0&0&0\\ \noalign{\medskip}0&0&0&0
\end {smallmatrix} \right].$$

To begin with we consider the $(1,2)$-entry of $[E_3,E_5]=E_3$ which gives $\beta(g-c-1)=0$. We separate cases according as first of all $\beta=0$ and then  $g-c-1=0$. Looking at the $(3,4)$-entry of $[E_3,E_5]=E_3$ we must have that $\mu\neq0$ or $E_1=0$ and so $t=j+1$. Next, from $[E_4,E_5]=bE_4$, assuming that $b\neq1$ we find that $s=0$. From the $(2,4)$-entry of $[E_2,E_4]=0$ we find that $q=0$. Now comparing $E_1$ and $[E_1,E_5]-2E_1$ we find that $j=c+1$ in order not to have $E_1=0$. Next from the $(1,3)$ and $(1,4)$-entries of $[E_2,E_4]=0$, respectively, we have $\sigma\neq0$ or $\mu\neq0$ or else $E_1=0$. From the $(2,3)$, $(2,4)$ and $(1,4)$-entries of $[E_3,E_5]=E_3$, respectively, we have that $g=c$, $\tau=h\mu-\sigma k$ and $\rho=dh\mu-\delta k+e\mu)$.
Looking at the $(2,4)$-entry of $[E_4,E_5]=bE_4$ we have $b=2$ and $r\neq0$ or else $E_1$ and $E_4$ are proportional. Now from the $(1,4)$-entry of $[E_4,E_5]=bE_4$ we have that $d=0$ and hence $E_3=0$.

Now we consider the second case where $g=c+1$ and $\beta\neq0$. From the $(1,2)$-entry of $[E_4,E_5]-bE_4=0$ we find that $m=0$ and from the $(1,3)$-entry of $[E_2,E_4]=0$ we find that $q=0$. Comparing $E_1$ and $[E_1,E_5]-2E_1$ we find that $t=c+2$ in order not to have $E_1=0$. Now not both $\mu$ and $\sigma$ can be zero or else $E_1=0$ so we distinguish subcases according as $\sigma$ and $\mu$ zero. So if $\sigma=0$ and $\mu\neq0$ the $(1,2)$ and $(3,4)$-entries of $[E_4,E_5]=bE_4$, respectively, give that $n=s=0$, since we are assuming that $b\neq1$. Finally the $(1,4)$-entry of $[E_2,E_4]$ implies that $r$ and hence  $E_1$ and $E_4$ are proportional. On the other hand if $\mu=0$ and $\sigma\neq0$ then from the $(1,3)$-entry of $[E_2,E_3]-E_1=0$ we find that $j=c+2$ and from the $(2,4)$-entry of $[E_2,E_4]=0$ that $s=0$. Next the $(2,4)$-entry of $[E_4,E_5]-bE_4=0$ that $r=0$, from the $(1,3)$-entry of $[E_3,E_5]-E_3=0$ that $\delta=d\sigma-\beta h$ from the $(1,4)$-entry of $[E_3,E_5]-E_3=0$ that $\rho=(k-i)\beta h+d\tau$. From the $(1,3)$-entry of $[E_4,E_5]-bE_4=0$ we find that $b=2$ and $n\neq0$ or else $E_1$ and $E_4$ are proportional. Finally the $(1,4)$-entry of $[E_4,E_5]-bE_4=0$ implies that $k=0$ which gives that $E_1=0$.




\begin{proposition} Algebra $A_{5,31}$ has no representation in
$\g\l(4,\R)$. \end{proposition}
\begin{proof} Put $$E_4=\left[\begin{smallmatrix}0&t&u&v\\ \noalign{\medskip}0&0&p&q\\
\noalign{\medskip}0&0&0&r\\ \noalign{\medskip}0&0&0&0\end
{smallmatrix}
 \right],
E_5=\left[ \begin{smallmatrix}a&b&c&d\\ \noalign{\medskip}0&e&f&g\\
\noalign{\medskip}0&0&h&i\\ \noalign{\medskip}0&0&0&j\end
{smallmatrix}
 \right].$$
\noindent and define $E_3=[E_4,E_5]-E_4, E_2=[E_3,E_4],
E_1=[E_2,E_4]$. Then

 $$E_1= \left[ \begin {smallmatrix} 0&0&0&tpr \left( 3\,e-a-3\,h+j \right)
\\ \noalign{\medskip}0&0&0&0\\ \noalign{\medskip}0&0&0&0
\\ \noalign{\medskip}0&0&0&0\end {smallmatrix} \right],
 E_2=\left[ \begin {smallmatrix} 0&0&tp \left( 2\,e-a-h \right) &2\,tqe-
tqa+2\,rtf+2\,ruh-rau-rbp-tpi-tqj-ruj\\ \noalign{\medskip}0&0&0&pr
 \left( 2\,h-e-j \right) \\ \noalign{\medskip}0&0&0&0
\\ \noalign{\medskip}0&0&0&0\end {smallmatrix} \right],$$
$$E_3= \left[ \begin {smallmatrix} 0&t \left( e-a-1 \right) &tf+uh-au-bp-u
&tg+ui+vj-av-bq-cr-v\\ \noalign{\medskip}0&0&p \left( h-e-1 \right)
&pi+qj-eq-fr-q\\ \noalign{\medskip}0&0&0&r \left( j-h-1 \right)
\\ \noalign{\medskip}0&0&0&0\end {smallmatrix} \right].$$
Of the ten brackets, three, $[E_2,E_4]=E_1, [E_3,E_4]=E_1, [E_4,E_5]=E_3+E_4$, are satisfied by construction. Furthermore
$[E_1,E_2]=0,[E_1,E_3]=0,[E_1,E_4]=0$ identically. Of the remaining four brackets we shall need only the following:
$$[E_1,E_5]-3E_1=\left[ \begin {smallmatrix} 0&0&0&tpr \left( 3\,e-a-3\,h+j \right)(j-a-3)
\\ \noalign{\medskip}0&0&0&0\\ \noalign{\medskip}0&0&0&0
\\ \noalign{\medskip}0&0&0&0\end {smallmatrix} \right],
[E_3,E_5]-E_3= \left[ \begin {smallmatrix} 0&t(e-a-1)^2&*&*
\\ \noalign{\medskip}0&0&p(h-e-1)^2&*\\ \noalign{\medskip}0&0&0&r(j-h-1)^2
\\ \noalign{\medskip}0&0&0&0\end {smallmatrix} \right].$$
Considering $E_1$, we see that $prt\neq0$. Then looking at $[E_3,E_5]-E_3=0$ we see that we must have $e=a+1,h=e+1,j=h+1$ and hence
$3e-a-3h+j=0$ and hence $E_1=0$ which is a contradiction.
\end{proof}


\begin{proposition} Algebra $A_{5,38}$ has no representation in
$\g\l(4,\R)$ \end{proposition}
\begin{proof} Define $\{E_1,E_2,E_4,E_5\}$ as
\[ \begin{array}{llll} & E_1=\left[\begin{smallmatrix} 0&\alpha&\beta&\delta\\ \noalign{\medskip}0&0&\epsilon&\phi
\\ \noalign{\medskip}0&0&0&\psi\\ \noalign{\medskip}0&0&0&0
 \end{smallmatrix}\right],\,
E_2=\left[\begin{smallmatrix} 0&\rho&\sigma&\tau\\ \noalign{\medskip}0&0&\lambda&\mu
\\ \noalign{\medskip}0&0&0&\theta\\0&0&0&0\end{smallmatrix}\right] ,\, E_4=\left[\begin{smallmatrix} a&b&c&d
\\ \noalign{\medskip}0&e&f&g\\ \noalign{\medskip}0&0&
h&i\\ \noalign{\medskip}0&0&0&j
\end{smallmatrix}\right],
E_5=\left[\begin{smallmatrix} m&n&p&q\\
0&r&s&u\\ 0&0&t&v\\ 0&0&0&w\end{smallmatrix}\right].
\end{array}\]
Then define $E_3=[E_4,E_5]$ and we must have that $[E_1,E_4]=E_1, [E_2,E_5]=E_2$ and the remaining seven commutators are zero.
Consulting the $(3,4)$ entries of $[E_3,E_4]$ and $[E_3,E_5]$ we deduce that the $(3,4)$th entry of $E_3$ vanishes. Likewise the $(1,2)$ and  $(2,3)$ entries of $E_3$ vanish. 

Just as in $A_{5,22}$ there are seven normal forms for $E_2$ under change of basis. We work through each of these cases again suppressing the details of the very complicated matrices concerned. In the first case $\rho=1, \lambda=1, \theta=1, \sigma=0, \tau=0, \mu=0$
From $[E_2,E_5]$ we must have that $r=m+1, s=n, u=p, t=r+1, v=s, w=t+1$ and from $[E_1,E_2]$ that
$\epsilon=\alpha, \phi=\beta, \psi=\epsilon$. From $[E_1,E_5]$ we have that $\alpha=0, \beta=0, \delta=0$ but now $E_1=0$.

In the second case (ii) $\rho=1, \lambda=1, \theta=0, \sigma=0, \tau=0, \mu=0$. Then $[E_2,E_5]=0$ implies that
$r=m+1, s=n, u=0, v=0, t=r+1$ and $[E_1,E_2]=0$ that $\epsilon=\alpha, \phi=0, \psi=0$. Now $[E_1,E_5]=0$ easily gives
$\alpha=0, \beta=0$ and $[E_2,E_4]=0$ that $e=a, f=b, g=0, i=0, h=e$. Finally the $(1,3)$th entry of $[E_3,E_5]$ implies that $c=0$ and then $E_1$ and $E_3$ are proportional.

In the third case  $\rho=1, \lambda=0, \theta=1, \sigma=0, \tau=0, \mu=0$. Then $[E_2,E_5]=0$ implies that
$r=m+1, s=0, u=p, w=t+1$, from $[E_1,E_2]=0$ that $\epsilon=0, \phi=\beta$ and from $[E_1,E_5]=0$ that $\alpha=0, \psi=0$. Next from $[E_2,E_4]=0$ we find that
$e=a, f=0, g=c, j=h$ and from $[E_4,E_5]=E_3$ that
$b=0, i=0$. Now the $(1,4)$th entry of $[E_2,E_3]=0$ implies that
$ap+ct-mc-ph=0$. In order for $E_3$ not to vanish we must have that $cv+d-nc\neq0$ and then from the the $(1,4)$th entries of $[E_3,E_4]$ and $[E_3,E_5]$
$h=a, t=m-1$. Finally $[E_4,E_5]=E_3$ implies that $c=0, d=0$ and hence $E_3=0$.

In the fourth case $\rho=0, \lambda=0, \theta=0, \sigma=1, \tau=0, \mu=1$
From $[E_2,E_5]=0$ we have that $t=m+1, v=n, w=r+1$, from $[E_1,E_2]=0$ that
$\psi=\alpha$ and from $[E_2,E_4]=0$ that $h=a, i=b, j=e$. Now from $[E_3,E_5]=0$ we find that
$eu+nf+gm+g-rg-bs-ua=0$ which can be used to simplify $E_3$. The $(1,3)$th entry of $[E_3,E_5]=0$ implies that
$c=nf-bs$. Now $bu-nbs+n^2f+d-ng-pb\neq0$ or else $E_3=0$ and then from the the $(1,4)$ entries of $[E_3,E_4]$ and  $[E_3,E_5]$ we deduce that
$e=a$ and $r=m-1$. Now from $[E_4,E_5]=E_3$ we have that $b=0, d=0, f=0, g=0$ which implies that $E_3=0$.

In the fifth case $\rho=0, \lambda=0, \theta=0, \sigma=1, \tau=0, \mu=0$. From $[E_2,E_5]=0$ we have that $t=m+1, v=0$, from $[E_1,E_2]=0$ that
$\psi=0$ and from $[E_2,E_4]=0$ that $h=a, i=0$. Now the $(1,3)$th entries of $[E_3,E_5]=0$ and $[E_1,E_5]=0$ implies that $c=nf-bs$ and
$\beta=n\epsilon-\alpha s$. Now we separate cases according as the $(2,4)$th entry $eu+fv+gm+g-rg-si-ua$ of $E_3$ vanishes or not. In the first of these subcases the $(1,4)$th entries of $[E_3,E_4]=0, [E_3,E_5]=0, [E_4,E_5]=E_3$ and $[E_1,E_4]=0$ imply that $j=a, w=m, d=0$ and $\delta=g\alpha-b\phi$. Now considering
$[E_1,E_4]=0$ we must have $e=a+1$ or $e=a-1$ or else $E_1=0$. However, now in either if these cases $[E_4,E_5]=E_3$ implies that $E_3=0$. Thus we may assume that $eu+fv+gm+g-rg-si-ua\neq0$ and then the $(1,4)$th entries of $[E_1,E_3]=0$ and $[E_3,E_5]=0$ imply that $\alpha=0$ and $n=0$. Then the $(1,2)$ and $(2,4)$ entries of $[E_3,E_4]=0$ give $b=0$ and $e=a$ and finally the $(2,4)$th entry of $[E_3,E_5]=0$ that $r=m$ and at this point we have that $[E_4,E_5]=E_3$ implies that  $E_3=0$.

In the sixth case $\rho=0, \lambda=1, \theta=0, \sigma=0, \tau=0, \mu=0$. From $[E_2,E_5]=0$ we have that $t=r+1, n=0, v=0$, from $[E_1,E_2]=0$ that
$\alpha=0, \psi=0$ and from $[E_2,E_4]=0$ that $b=0, h=e, i=0$. Now $[E_4,E_5]=E_3$ imply that $f=0$ and the $(2,3)$th entry of $[E_1,E_4]=E_1$ that $\epsilon=0$. Now the $(1,3)$th entries of $[E_3,E_4]=0$ and $[E_3,E_5]=0$ imply that the $(1,3)$th entry $ap+cr+c-mc-pe=0$ of $E_3$ vanishes. Suppose that $d\neq0$. Then the $(1,4)$ entries of $[E_3,E_4]=0, [E_3,E_5]=0$ and $[E_4,E_5]=E_3$ imply that $j=a, w=m, d=0$. Hence $d=0$. Now comparing the $(2,4)$th entries of $[E_3,E_4]=0$ and $E_3$ we find that $j=e$ and from the $(2,4)$th entry of $[E_3,E_5]=0$ that $w=r$. Now, however, the $(2,4)$th entry of $[E_4,E_5]=E_3$ implies that $E_3=0$.

In the seventh case $\rho=0, \lambda=0, \theta=0, \sigma=0, \tau=1, \mu=0$., consulting the $(3,4)$ entries of $[E_3,E_4]$ and $[E_3,E_5]$ we deduce that the $(3,4)$th entry of $[E_3,E_4]$ vanish. Likewise the $(1,2)$ and  $(2,3)$ entries of $[E_3,E_4]$ vanish. 
Now we cannot have that both the $(1,3)$ and $(2,4)$ entries of $E_3$ vanish or else $E_2$ and $E_3$ would be proportional. Appealing to Corollary 3.2 we may assume that the $(1,3)$th entry of $E_3$ is not zero. Hence from the $(1,2)$ entries of $[E_3,E_4]$ and $[E_3,E_5]$ we deduce that $h=a$ and $t=m$. Now the $(1,2)$ entries of $[E_1,E_5]$ and $[E_4,E_5]$ imply that $\psi=0$ and $i=0$. Next the $(1,3)$  and $(1,4)$ entries of $[E_1,E_4]$ give that $\beta=f\alpha -b\epsilon$ and $\delta=g\alpha+fi\alpha-bi\epsilon-b\phi$ and the $(1,4)$th entry of $[E_3,E_5]$ implies that $d=-(vbs-2vnf+bu+cv-ngm-2ng-pi-neu+nrg+nsi+nua)$.
Consulting the $(1,2)$ and $(2,3)$ entries of $[E_1,E_4]=E_1$ we see that at least one of $\alpha$ and $\epsilon$ must be zero; hence from the $(1,2)$ and $(2,3)$ entries of $[E_1,E_5]$ we deduce that $r=m$. Next the $(2,4)$ entries of $[E_3,E_5]$ and $[E_1,E_5]$ imply that $g=-(eu+fv-si-ua)$ and $\phi=-v\epsilon$.

Now we separate cases assuming first of all that $\alpha=0$. Then $\epsilon\neq0$ or else $E_1=0$. Now the $(1,3)$ entry of $[E_1,E_5]$ implies that $n=0$
and the $(2,3)$ entry of $[E_1,E_4]-E_1$ and $[E_4,E_5]-E_3$ implies that $s=0$. Now, however, $E_3=0$, a contradiction.

Finally suppose that $\epsilon=0$ and $\alpha\neq0$.  Then the $(1,2)$ entry of $[E_1,E_4]-E_1$ gives that $e=a+1$. The $(1,2)$ and $(2,3)$ entries of $[E_4,E_5]-E_3$ give $n=s=0$ Now, however, again $E_3=0$.
\end{proof}

Since $A_{5,39}$  is equivalent over $\C$ to $A_{5,38}$ we deduce that
\begin{corollary} $A_{5,39}$ has no representation in $\g\l(4,\R)$. \end{corollary}


\begin{thebibliography}{99}




\bibitem{GT} Ghanam, R. and Thompson, G., {\it Minimal Matrix representations of Four-Dimensional Lie
algebras}, Bull. Malays. Math. Sci. Soc. ({\bf 2}) 36(2), 2013, 343-349.

\bibitem{GST} Ghanam,R., Strugar, I.  and Thompson,G.(2005) {\it Matrix representations for Low Dimensional Lie
algebras}, Extracta Mathematicae, Vol {\bf{20}(2)},151-184.

\bibitem {Hum} Humphreys, J., {\it Lie Algebras and their
Representations}, Springer (1997).
%
%

\bibitem {Jac} Jacobson, N., {\it Schur's Theorems on commutative matrices}, Bull. Amer. Math. Soc.  {\bf{50}} (1946), 431-436.

\bibitem{KB} Kang Y. and Bai C., (2008),  {\it Refinment of Ado's Theorem
in Low Dimensions and Applications in Affine Geometry},
Communications in Algebra,  {\bf 36}(1), 82-93.

\bibitem{Mub1} Mubarakzyanov, G.M., (1963) {\it  Classification of real
Lie Algebras in dimension five}, Izv. Vysshikh Uchebn. Zavedneii
Mat. {\bf 3(34)}, 99-106.

\bibitem{PSWZ} Patera, J., Sharp, R.T., Winternitz, P. and Zassenhaus, H., (1976) {\it Invariants of real low dimension
Lie algebras}, J. Math. Phys. {\bf 17}, 986-994.

\bibitem{RT} M. Rawashdeh and G. Thompson,  {\it The inverse problem
for six-dimensional codimension two nilradical Lie algebras}, J.
Math. Phys., 2006, {\bf 47}(11), 112901.








\end{thebibliography}
\end{document}